\documentclass[10pt,a4paper,fleqn]{article}

\usepackage[T1]{fontenc}
\usepackage[utf8]{inputenc}
\usepackage[english]{babel}

\usepackage[a4paper,margin=3cm]{geometry}

\usepackage{amsmath,amssymb,amsthm,mathtools}
\usepackage{mathrsfs}
\usepackage{bm}

\usepackage{xcolor}
\usepackage{graphicx}

\usepackage[numbers,sort&compress]{natbib}

\usepackage[colorlinks=true,citecolor=blue,linkcolor=blue,urlcolor=blue]{hyperref}
\usepackage[nameinlink,capitalize,noabbrev]{cleveref}

\usepackage{aliascnt}

\newtheorem{theorem}{Theorem}[section]

\newaliascnt{proposition}{theorem}
\newtheorem{proposition}[proposition]{Proposition}
\aliascntresetthe{proposition}

\newaliascnt{lemma}{theorem}

\aliascntresetthe{lemma}

\newaliascnt{corollary}{theorem}
\newtheorem{corollary}[corollary]{Corollary}
\aliascntresetthe{corollary}

\theoremstyle{definition}

\newaliascnt{definition}{theorem}
\newtheorem{definition}[definition]{Definition}
\aliascntresetthe{definition}

\newaliascnt{example}{theorem}
\newtheorem{example}[example]{Example}
\aliascntresetthe{example}

\newaliascnt{construction}{theorem}

\aliascntresetthe{construction}

\newaliascnt{notation}{theorem}

\aliascntresetthe{notation}

\newaliascnt{remark}{theorem}
\newtheorem{remark}[remark]{Remark}
\aliascntresetthe{remark}

\Crefname{theorem}{Theorem}{Theorems}

\Crefname{proposition}{Proposition}{Propositions}

\Crefname{lemma}{Lemma}{Lemmas}

\Crefname{corollary}{Corollary}{Corollaries}

\Crefname{definition}{Definition}{Definitions}

\Crefname{example}{Example}{Examples}

\Crefname{construction}{Construction}{Constructions}

\Crefname{notation}{Notation}{Notations}

\Crefname{remark}{Remark}{Remarks}

\newcommand{\R}{\mathbb{R}}
\newcommand{\C}{\mathbb{C}}
\newcommand{\CP}{\mathbb{CP}}
\newcommand{\Hbb}{\mathbb{H}}
\newcommand{\HC}{\mathbb{H}_{\C}}

\newcommand{\Sp}{\operatorname{Sp}}
\newcommand{\Spin}{\operatorname{Spin}}
\newcommand{\re}{\operatorname{Re}}
\newcommand{\im}{\operatorname{Im}}

\newcommand{\ii}{\mathrm{i}}
\newcommand{\dd}{\mathrm{d}}
\newcommand{\dz}{\mathrm{d}z}

\newcommand{\inner}[2]{\langle #1,#2\rangle}
\newcommand{\norm}[1]{\left|#1\right|}

\newcommand{\NullCone}{\mathcal{A}}

\title{Complex Quaternions and Superminimal Surfaces in Four-Space}

\author{
Amedeo Altavilla\thanks{Dipartimento di Matematica, Universit\`a degli Studi di Bari Aldo Moro, Via E. Orabona 4, 70125 Bari, Italy. Email: \href{mailto:amedeo[.altavilla@uniba.it](mailto:.altavilla@uniba.it)}{[amedeo.altavilla@uniba.it](mailto:amedeo.altavilla@uniba.it)}. Corresponding author.}
\and
Elena Giorgio\thanks{Dipartimento di Matematica, Universit\`a degli Studi di Bari Aldo Moro, Via E. Orabona 4, 70125 Bari, Italy. Email: \href{mailto:e[.giorgio5@studenti.uniba.it](mailto:.giorgio5@studenti.uniba.it)}{[e.giorgio5@studenti.uniba.it](mailto:e.giorgio5@studenti.uniba.it)}.}
}

\date{}

\begin{document}

\maketitle

\begin{abstract}
We develop a quaternionic approach to conformal superminimal surfaces
in Euclidean four-space. The starting point is the classical
Weierstrass representation: every conformal minimal immersion
$X\colon M\to\R^4$ is recovered as
$X = c + \re\int\Phi,\dz$,
where $\Phi$ is a holomorphic null curve in $\C^4$, identified with
the algebra of complex quaternions $\HC$. The multiplicativity of the
quaternionic symmetrized norm makes it natural to factor null curves as
$\Phi=ALB$, where $A$ and $B$ are holomorphic maps of unit symplectic
norm and $L$ is a holomorphic null element. We show that on simply
connected domains the null factor can always be taken constant.

The main result is an explicit quaternionic reformulation of the
superminimality condition --- the requirement that the curvature
ellipse be a circle at every point. In the fixed-null gauge
$L=1+\ii e_1$, superminimality is equivalent to the vanishing of a
product of two holomorphic functions built from the left and right
Maurer--Cartan forms of $A$ and $B$. On a connected domain this
forces one of the two components of the generalized Gauss map
$[\Phi]\colon M\to Q^2\simeq\CP^1\times\CP^1$ to be constant,
recovering in spinorial terms the classical ruling condition on the
projective null quadric.

We further provide a first-order ODE parametrization of the
superminimal $ALB$-data, analyse the residual gauge freedom, prove a
fixed-gauge rigidity statement for polynomial spinorial factors, and
illustrate the theory with explicit examples.
\end{abstract}

\medskip


\medskip


\medskip

\section{Introduction}
\label{sec:introduction}

Superminimal surfaces form a distinguished class of submanifolds within four-dimensional
Riemannian geometry.  In an oriented four-manifold they may be
characterised by the circularity of the curvature ellipse, or
equivalently by the condition that the associated Gauss map lies in one
of the two rulings of the projective null quadric.  The terminology was
introduced by Bryant in his study of conformal immersions into the
four-sphere, though the underlying geometry goes back to earlier work of
Kommerell and was subsequently studied by Eisenhart, Bor\r{u}vka,
Calabi, Chern, and others (see
\cite{Bryant1982,Friedrich1997,Forstneric2020CY,AlarconForstnericLarusson2021}).
We work in Euclidean four-space~$\R^4$, where a superminimal surface is
a conformal minimal immersion whose curvature ellipse is a circle at
every point.

Let $M\subset\C$ be a simply connected domain.  We recall that every conformal minimal
immersion $X\colon M\to\R^4$ is given by the Weierstrass representation
\[
  X = c + \re\int\Phi\,\dz,
\]
where $\Phi=2X_z$ is a nowhere-vanishing holomorphic null curve, that
is, a holomorphic map $\Phi\colon M\to\C^4$ satisfying
$\Phi^s:=\inner{\Phi}{\Phi}=0$ (see
\cite{Osserman1986,HoffmanOsserman1980,AlarconForstnericLopez2021}).
The superminimality condition admits a concise reformulation in these
terms: let $\Phi'$ denotes the complex derivative of $\Phi$, then $X$ is superminimal if and only if $(\Phi')^s=0$.  Projectively,
this means that the generalised Gauss map
$[\Phi]\colon M\to Q^2\simeq\CP^1\times\CP^1$ lies in a fibre of one
of the two projections, i.e., at least one of the two components
$g_1,g_2\colon M\to\CP^1$ of $[\Phi]$ is constant.  This is the
Euclidean counterpart of Bryant's twistor correspondence for
superminimal surfaces in~$S^4$, extended to general oriented
four-manifolds by Friedrich; see
\cite{Bryant1982,EellsSalamon1985,Friedrich1997}.

The purpose of this paper is to rewrite the above picture in the
language of complex quaternions.  We identify
$\C^4\simeq\HC:=\Hbb\otimes\C$, equipping it with the symmetrised norm
$q^s=qq^c$.  It holds that
$(qr)^s=q^sr^s$.  This makes it natural to factor any holomorphic null
curve as
\[
  \Phi = ALB,
\]
where $A,B\colon M\to\Sp(1,\C)\simeq SL(2,\C)$ are holomorphic maps of
unit symplectic norm and $L$ is a holomorphic null element.  We prove
(\Cref{thm:global-ALB-representation}) that on simply connected domains
the null factor $L$ may always be taken constant --- the
\emph{fixed-null gauge} --- so the representation captures all
conformal minimal immersions without loss of generality.  

\smallskip
\noindent\textbf{Main results.}
Differentiating $\Phi=ALB$ and introducing the left and right
Maurer--Cartan forms $\alpha=A^cA'$ and $\beta=B'B^c$ of the spinorial
factors, the condition $(\Phi')^s=0$ reduces, in the fixed-null gauge
$L=1+\ii e_1$, to the factored product condition
\[
  (\alpha_2+\ii\alpha_3)(\beta_2-\ii\beta_3)=0,
\]
where $\alpha_k$, $\beta_k$ are the $e_k$-components of $\alpha$ and
$\beta$ (\Cref{thm:superminimal-ALB-normal-form}).  On a connected
domain, one of the two factors must vanish identically, corresponding
to the two rulings of $Q^2\simeq\CP^1\times\CP^1$.  Each alternative
admits an explicit first-order ODE parametrisation inside a solvable
two-dimensional Lie algebra (\Cref{prop:ODE-parametrization}),
reminiscent of first-order descriptions of twistor lifts of
superminimal surfaces in spheres \cite{ChiFernandezWu1999}.  For
polynomial spinorial factors, the same condition implies a gauge
rigidity: modulo the exact stabiliser of~$L$, one of $A$ or $B$ may be
chosen constant (\Cref{prop:polynomial-normal-form}).

\smallskip
\noindent\textbf{Organisation of the paper.}
\Cref{sec:preliminaries} recalls conformal minimal immersions in~$\R^4$,
the curvature ellipse, and the superminimality criterion $(\Phi')^s=0$,
together with the projective null quadric and the generalised Gauss map.
\Cref{sec:complex-quaternions} introduces complex quaternions and the
$ALB$-representation.
\Cref{sec:superminimal-ALB} derives the moving and fixed-gauge
superminimality conditions, proves the factored formula, and provides
the ODE parametrisation and the associate family.
\Cref{sec:polynomial} contains the polynomial rigidity statement.
\Cref{sec:examples} gives explicit examples.

\section{Preliminaries on Minimal and Superminimal Surfaces}
\label{sec:preliminaries}

This section fixes the differential-geometric conventions used
throughout the paper. We recall conformal minimal surfaces in $\R^4$
and the characterisation of superminimality via the additional nullity
condition $(\Phi')^s=0$ on the associated holomorphic null curve, and
the projective interpretation through the generalised Gauss map
$[\Phi]\colon M\to Q^2\simeq\CP^1\times\CP^1$, whose two rulings
correspond to the two possible superminimal orientations. Standard
references for minimal surfaces in Euclidean spaces are
\cite{Osserman1986,HoffmanOsserman1980}; for superminimal surfaces and
their twistor interpretation in four-dimensional space forms we refer
to
\cite{Bryant1982,EellsSalamon1985,Friedrich1997,DajczerTojeiro2009,
Forstneric2020CY}.

\subsection{Conformal immersions}
\label{subsec:conformal-immersions}

Throughout this paper $M\subset\C$ denotes a connected and simply
connected domain with complex coordinate $z=u+\ii v$, and
$X\colon M\to\R^4$ a smooth immersion. We write
\[
  X_z := \tfrac{1}{2}(X_u - \ii X_v),
  \qquad
  X_{\bar z} := \tfrac{1}{2}(X_u + \ii X_v),
\]
and extend the Euclidean inner product of $\R^4$ $\C$-bilinearly to
$\C^4$.
{The \emph{first fundamental form} of the immersion is the metric
induced on $M$ by the Euclidean metric of $\R^4$. In the local
coordinates $(u,v)$ it is
$$
  \mathrm{I}
  =
  E\,\dd u^2+2F\,\dd u\,\dd v+G\,\dd v^2,
$$
where
$E:=\inner{X_u}{X_u},
  F:=\inner{X_u}{X_v},
  G:=\inner{X_v}{X_v}.
$}

The immersion $X$ is called \emph{conformal} if
$E=G$ and $F=0$.
Writing $e^{2\omega}:=\inner{X_u}{X_u}$, conformality is equivalent to
\begin{equation}
\label{eq:conformality-complex}
  \inner{X_z}{X_z}=0,
  \qquad
  \inner{X_z}{X_{\bar z}}=\tfrac{1}{2}e^{2\omega}>0.
\end{equation}
The first condition says that $\Phi:=2X_z$ is a null vector in $\C^4$;
the second is the immersion condition.

\subsection{The second fundamental form and the curvature ellipse}
\label{subsec:curvature-ellipse}

For $p\in M$, let $N_pM:=(\dd X_p(T_pM))^\perp\subset\R^4$ be the
normal plane. The \emph{second fundamental form} is the symmetric
bilinear map
\[
  \mathrm{II}\colon T_pM\times T_pM\longrightarrow N_pM,
  \qquad
  \mathrm{II}(V,W):=(\nabla_V W)^\perp,
\]
where $\nabla$ is the ambient Euclidean connection and
$(\,\cdot\,)^\perp$ denotes projection onto $N_pM$. 
{
In the Euclidean situation considered here, the ambient connection is
the ordinary derivative in $\R^4$. Hence the definition becomes very
concrete. Writing $(x_1,x_2)=(u,v)$, one has
$$
  \mathrm{II}(X_{x_i},X_{x_j})
  =
  \bigl(X_{x_i x_j}\bigr)^\perp.
$$
Equivalently,
$$
  \mathrm{II}(X_u,X_u)=\bigl(X_{uu}\bigr)^\perp,
  \qquad
  \mathrm{II}(X_u,X_v)=\bigl(X_{uv}\bigr)^\perp,
  \qquad
  \mathrm{II}(X_v,X_v)=\bigl(X_{vv}\bigr)^\perp.
$$
Thus the second fundamental form records the normal part of the second
derivatives of the parametrization.

Since the surface lies in $\R^4$, the normal space $N_pM$ is
two-dimensional. Consequently, the coefficients of $\mathrm{II}$ are
 vectors in the normal plane. If
$(n_1,n_2)$ is a local orthonormal frame of the normal bundle, then
$$
  \mathrm{II}(X_{x_i},X_{x_j})
  =
  h_{ij}^1 n_1+h_{ij}^2 n_2,
$$
where
$$
  h_{ij}^\alpha
  :=
  \inner{X_{x_i x_j}}{n_\alpha},
  \qquad
  \alpha=1,2.
$$
Therefore $\mathrm{II}$ may be viewed as a pair of scalar symmetric
forms, one for each normal direction.
}

In the conformal
frame $\varepsilon_1:=e^{-\omega}X_u$, $\varepsilon_2:=e^{-\omega}X_v$
one has
\begin{equation}
\label{eq:II-conformal}
  \mathrm{II}(\varepsilon_i,\varepsilon_j)
  \;=\;
  e^{-2\omega}(X_{x_ix_j})^\perp,
  \qquad (x_1,x_2):=(u,v).
\end{equation}
The \emph{mean curvature vector} is the trace
\begin{equation}
\label{eq:mean-curvature}
  \mathbf{H}
  \;:=\;
  \tfrac{1}{2}\bigl[
    \mathrm{II}(\varepsilon_1,\varepsilon_1)
    +\mathrm{II}(\varepsilon_2,\varepsilon_2)
  \bigr]
  \;=\;
  \tfrac{1}{2}e^{-2\omega}
  \bigl(X_{uu}+X_{vv}\bigr)^\perp.
\end{equation}

\begin{remark}
\label{rem:minimal-harmonic}
In conformal coordinates the tangential component of
$X_{uu}+X_{vv}$ vanishes identically, so
$X_{uu}+X_{vv}=2e^{2\omega}\mathbf{H}$. Hence $X$ is
\emph{minimal} (i.e.\ $\mathbf{H}=0$) if and only if $X$ is
\emph{harmonic}:
\begin{equation}
\label{eq:minimal-harmonic}
  X_{uu}+X_{vv}=0,
\end{equation}
Thus, in conformal coordinates, minimal immersions are precisely
harmonic conformal immersions.
\end{remark}

Setting
\begin{equation}
\label{eq:xi-eta}
  \xi
  \;:=\;
  \frac{\mathrm{II}(\varepsilon_1,\varepsilon_1)
        -\mathrm{II}(\varepsilon_2,\varepsilon_2)}{2},
  \qquad
  \eta \;:=\; \mathrm{II}(\varepsilon_1,\varepsilon_2),
\end{equation}
one has
$\mathrm{II}(\varepsilon_1,\varepsilon_1)=\mathbf{H}+\xi$ and
$\mathrm{II}(\varepsilon_2,\varepsilon_2)=\mathbf{H}-\xi$.
For a unit tangent vector
$w=\cos\theta\,\varepsilon_1+\sin\theta\,\varepsilon_2$,
bilinearity and symmetry of $\mathrm{II}$ give
\[
  \mathrm{II}(w,w)
  =
  \cos^2\theta\,\mathrm{II}(\varepsilon_1,\varepsilon_1)
  +2\cos\theta\sin\theta\,\mathrm{II}(\varepsilon_1,\varepsilon_2)
  +\sin^2\theta\,\mathrm{II}(\varepsilon_2,\varepsilon_2).
\]
Substituting the standard identities
$\cos^2\theta=\frac{1+\cos2\theta}{2}$,
$\sin^2\theta=\frac{1-\cos2\theta}{2}$,
$2\cos\theta\sin\theta=\sin2\theta$, and regrouping, we obtain
\begin{align*}
  \mathrm{II}(w,w)
  &=
  \frac{\mathrm{II}(\varepsilon_1,\varepsilon_1)
       +\mathrm{II}(\varepsilon_2,\varepsilon_2)}{2}
  +\cos(2\theta)\,
  \frac{\mathrm{II}(\varepsilon_1,\varepsilon_1)
       -\mathrm{II}(\varepsilon_2,\varepsilon_2)}{2}
  +\sin(2\theta)\,\mathrm{II}(\varepsilon_1,\varepsilon_2).
\end{align*}
Recognising $\mathbf{H}$, $\xi$, and $\eta$ from
\Cref{eq:mean-curvature,eq:xi-eta} gives
\begin{equation}
\label{eq:curvature-ellipse}
  \mathrm{II}(w,w)
  \;=\;
  \mathbf{H} + \xi\cos(2\theta) + \eta\sin(2\theta).
\end{equation}
As $\theta$ varies, $\mathrm{II}(w,w)$ traces an ellipse in
$N_pM$ centred at $\mathbf{H}$, called the
\emph{curvature ellipse} $\mathcal{E}_p$. It is a circle if and
only if
\begin{equation}
\label{eq:curvature-ellipse-circle}
  \inner{\xi}{\eta}=0
  \qquad\text{and}\qquad
  \norm{\xi}=\norm{\eta}.
\end{equation}

\begin{definition}
\label{def:superconformal-superminimal}
A conformal immersion $X\colon M\to\R^4$ is called
\emph{superconformal} if its curvature ellipse $\mathcal{E}_p$ is a
circle at every $p\in M$, and \emph{superminimal} if it is both
minimal and superconformal.
\end{definition}

The circularity condition \Cref{eq:curvature-ellipse-circle}
is related to the Wintgen inequality
$K+|K_N|\leq\norm{\mathbf{H}}^2$ for surfaces in $\R^4$, with
equality precisely in the superconformal case; see
\cite{Wintgen1979,GuadalupeRodriguez1983,DajczerTojeiro2009}.

\begin{remark}
\label{rem:superminimal-terminology}
Some authors use ``superminimal'' primarily for surfaces in $S^4$,
in connection with twistor theory. Here we follow the Euclidean
convention of \cite{DajczerTojeiro2009}.
\end{remark}
\subsection{Holomorphic null curves and the Weierstrass representation}
\label{subsec:holomorphic-null-curves}

Let $X\colon M\to\R^4$ be a conformal immersion and set
$\Phi:=2X_z=X_u-\ii X_v$. Writing $\Phi^s:=\inner{\Phi}{\Phi}$ for the
complex-bilinear square norm, \Cref{eq:conformality-complex} translates
into
\begin{equation}
\label{eq:Phi-null-and-nonzero}
  \Phi^s=0,
  \qquad
  \inner{\Phi}{\overline{\Phi}}>0,
\end{equation}
where the second condition is equivalent to $\Phi$ being nowhere zero.
Since $\Phi_{\bar{z}}=2X_{z\bar{z}}$, minimality of $X$ is equivalent
to holomorphicity of $\Phi$.

Conversely, every holomorphic map $\Phi\colon M\to\C^4$ satisfying
$\Phi^s=0$ and $\Phi\neq 0$ is called a \emph{holomorphic null curve}.
Since $M$ is simply connected, the holomorphic $1$-form $\Phi\,\dz$
has a primitive, and
\begin{equation}
\label{eq:minimal-immersion-from-null-curve}
  X(z) = c + \re\!\int_{z_0}^{z}\Phi(\zeta)\,\dd\zeta,
  \qquad c\in\R^4,
\end{equation}
defines a conformal minimal immersion. If zeros of $\Phi$ are allowed,
the same formula yields a branched conformal minimal surface.
\subsection{Superminimality in terms of
  \texorpdfstring{$\Phi'$}{Phi'}}
\label{subsec:superminimal-Phi-prime}

Throughout the paper, for a holomorphic map $\Phi$ we write
$\Phi':=\partial\Phi/\partial z$. We now express the circularity of the
curvature ellipse directly in terms of $\Phi'$.

Let $X\colon M\to\R^4$ be a conformal immersion. Since $X_z=\tfrac{1}{2}(X_u-\ii X_v)$, differentiating once more
gives
\[
  X_{zz}
  \;=\;
  \frac{1}{4}\bigl(X_{uu}-X_{vv}-2\ii X_{uv}\bigr).
\]
Taking normal components and applying \Cref{eq:II-conformal}:
\[
  (X_{zz})^\perp
  \;=\;
  \frac{1}{4}\bigl(X_{uu}^\perp - X_{vv}^\perp - 2\ii X_{uv}^\perp\bigr)
  \;=\;
  \frac{e^{2\omega}}{4}
  \bigl[
    \mathrm{II}(\varepsilon_1,\varepsilon_1)
    -\mathrm{II}(\varepsilon_2,\varepsilon_2)
    -2\ii\,\mathrm{II}(\varepsilon_1,\varepsilon_2)
  \bigr].
\]
Substituting the expressions from \Cref{eq:xi-eta} gives
\begin{equation}
\label{eq:Xzz-normal}
  (X_{zz})^\perp
  \;=\;
  \frac{e^{2\omega}}{2}\,(\xi-\ii\eta).
\end{equation}
Hence
\[
  \inner{(X_{zz})^\perp}{(X_{zz})^\perp}
  =
  \frac{e^{4\omega}}{4}
  \bigl(\norm{\xi}^2-\norm{\eta}^2-2\ii\inner{\xi}{\eta}\bigr).
\]
By \Cref{eq:curvature-ellipse-circle}, the curvature ellipse is a
circle if and only if
\begin{equation}
\label{eq:isotropic-Xzz-normal}
  \inner{(X_{zz})^\perp}{(X_{zz})^\perp}=0.
\end{equation}

\begin{proposition}
\label{prop:superminimal-null-curve-condition}
Let $X\colon M\to\R^4$ be a conformal minimal immersion and set
$\Phi:=2X_z$. Then $X$ is superminimal if and only if
\begin{equation}
\label{eq:superminimal-Phi-prime}
  (\Phi')^s=\inner{\Phi'}{\Phi'}=0.
\end{equation}
\end{proposition}

\begin{proof}
Differentiating $\inner{X_z}{X_z}=0$ gives $\inner{X_{zz}}{X_z}=0$.
Writing $(X_{zz})^\top=aX_z+bX_{\bar z}$, the condition
$\inner{X_{zz}}{X_z}=0$ forces $b=0$, since $\inner{X_z}{X_z}=0$ and
$\inner{X_{\bar z}}{X_z}>0$. Thus $(X_{zz})^\top=aX_z$, and the
nullity of $X_z$ gives
$\inner{X_{zz}}{X_{zz}}=\inner{(X_{zz})^\perp}{(X_{zz})^\perp}$.
By minimality, $\Phi=2X_z$ is holomorphic, so $\Phi'=2X_{zz}$ and
\[
  \inner{\Phi'}{\Phi'}
  =
  4\inner{(X_{zz})^\perp}{(X_{zz})^\perp}.
\]
This vanishes if and only if the curvature ellipse is a circle, by
\Cref{eq:isotropic-Xzz-normal}.
\end{proof}

Thus a nowhere-vanishing holomorphic null curve $\Phi\colon M\to\C^4$
defines a superminimal conformal minimal immersion $X=\re\int\Phi\,\dz$
if and only if
\begin{equation}
\label{eq:superminimal-summary}
  \Phi^s=0,
  \qquad
  \Phi\neq0,
  \qquad
  (\Phi')^s=0.
\end{equation}
Allowing zeros of $\Phi$ yields branched superminimal surfaces.

\subsection{The projective null quadric and the Gauss map}
\label{subsec:projective-null-quadric-prelim}

The superminimality condition \Cref{eq:superminimal-summary} has a
clean projective interpretation; see
\cite{HoffmanOsserman1980,Osserman1986} for the generalised Gauss map
of minimal surfaces and
\cite{Bryant1982,EellsSalamon1985,DajczerTojeiro2009} for
superminimality and its twistor interpretation.

The \emph{projective null quadric} is
{
the projectivisation of the non-zero null vectors for the
complex-bilinear scalar product on $\C^4$, namely
$$
  Q^2
  :=
  \bigl\{[\xi]\in\CP^3:\inner{\xi}{\xi}=0\bigr\}.
$$
Writing $\xi=(\xi_0,\xi_1,\xi_2,\xi_3)$, this condition becomes
$
  \inner{\xi}{\xi}
  =
  \xi_0^2+\xi_1^2+\xi_2^2+\xi_3^2=0,
$
so equivalently
}
\[
  Q^2 := \{[\xi]\in\CP^3:\xi_0^2+\xi_1^2+\xi_2^2+\xi_3^2=0\}.
\]
We identify $Q^2\simeq\CP^1\times\CP^1$ via the Segre parametrisation
\[
\begin{aligned}
  \sigma\colon\CP^1\times\CP^1&\longrightarrow Q^2\subset\CP^3,\\
  \bigl([s_0:s_1],[t_0:t_1]\bigr)
  &\longmapsto
  \bigl[s_0t_0+s_1t_1,\;\ii(s_0t_0-s_1t_1),\;
        s_1t_0-s_0t_1,\;-\ii(s_1t_0+s_0t_1)\bigr],
\end{aligned}
\]
whose image lies in $Q^2$ by direct computation. The two rulings of
$Q^2$ are the images of the fibres of the two projections
$\CP^1\times\CP^1\to\CP^1$.

Let $X\colon M\to\R^4$ be a conformal minimal immersion with associated
null curve $\Phi=2X_z$. Since $\Phi\neq 0$, the projectivisation
$[\Phi]\colon M\to Q^2$ is well defined; this is the \emph{generalised
Gauss map} of $X$, realising the identification
$G_2^+(\R^4)\simeq Q^2\simeq\CP^1\times\CP^1$. We denote its two
local components by $g_1,g_2\colon M\to\CP^1$.

In the affine chart $s_0t_0\neq 0$, set $g_1=s_1/s_0$ and $g_2=t_1/t_0$;
the Segre parametrisation gives
$[\Phi]=[1+g_1g_2,\;\ii(1-g_1g_2),\;g_1-g_2,\;-\ii(g_1+g_2)]$,
so the null curve takes the form
\begin{equation}
\label{eq:null-cone-parametrization}
  \Phi = \mu\bigl(1+g_1g_2,\;\ii(1-g_1g_2),\;g_1-g_2,\;-\ii(g_1+g_2)\bigr),
\end{equation}
where $\mu=(\Phi_0-\ii\Phi_1)/2$ is holomorphic and nowhere zero.
Outside this affine chart, $g_1$ and $g_2$ may become meromorphic,
with poles compensated by zeros of $\mu$.

We now show that
\begin{equation}
\label{eq:Phi-prime-g1-g2}
  \inner{\Phi'}{\Phi'} = -4\mu^2\,g_1'g_2'.
\end{equation}
{
Indeed, set
$$
  V(g_1,g_2)
  :=
  \bigl(
    1+g_1g_2,\;
    \ii(1-g_1g_2),\;
    g_1-g_2,\;
    -\ii(g_1+g_2)
  \bigr),
$$
so that $\Phi=\mu V$. A direct computation gives
$\inner{V}{V}=0$.
Moreover,
$$
  V_{g_1}
  =
  \bigl(g_2,\;-\ii g_2,\;1,\;-\ii\bigr),
  \qquad
  V_{g_2}
  =
  \bigl(g_1,\;-\ii g_1,\;-1,\;-\ii\bigr).
$$
Hence
$
  \inner{V_{g_1}}{V_{g_1}}=0=
  \inner{V_{g_2}}{V_{g_2}}
$
and
$$
\begin{aligned}
  \inner{V_{g_1}}{V_{g_2}}
  &=
  g_1g_2+(-\ii g_2)(-\ii g_1)-1+(-\ii)(-\ii) \\
  &=
  g_1g_2-g_1g_2-1-1 \\
  &=
  -2.
\end{aligned}
$$
Since $\inner{V}{V}=0$, differentiating with respect to $g_1$ and
$g_2$ also gives
$
  \inner{V}{V_{g_1}}=0=
  \inner{V}{V_{g_2}}$.
Now
$$
  \Phi'
  =
  \mu'V+\mu\bigl(g_1'V_{g_1}+g_2'V_{g_2}\bigr).
$$
Using the identities above, all terms containing $\mu'V$ vanish in
the square norm, and therefore
$$
\begin{aligned}
  \inner{\Phi'}{\Phi'}
  &=
  \mu^2
  \inner{
    g_1'V_{g_1}+g_2'V_{g_2}
  }{
    g_1'V_{g_1}+g_2'V_{g_2}
  } \\
  &=
  2\mu^2 g_1'g_2' \inner{V_{g_1}}{V_{g_2}} \\
  &=
  -4\mu^2 g_1'g_2'.
\end{aligned}
$$
}
The superminimality condition $(\Phi')^s=0$ is therefore locally
equivalent to $g_1'g_2'=0$. Since $M$ is connected and $g_1',g_2'$ are
holomorphic, the identity principle gives
\begin{equation}
\label{eq:superminimal-gauss}
  g_1'\equiv0 \qquad\text{or}\qquad g_2'\equiv0.
\end{equation}
Geometrically, $[\Phi]$ lies in a fibre of one of the two projections:
one component of the generalised Gauss map is constant. For non-planar
surfaces, exactly one component is constant; the two alternatives
correspond to the two possible orientations of the circular curvature
ellipse, and are the geometric source of the dichotomy in the
$ALB$-representation. The relation with twistor orientations is
classical; see \cite{Bryant1982,EellsSalamon1985,Friedrich1997}.

\begin{proposition}[Local normal form of superminimal null curves]
\label{prop:classical-superminimal-normal-form}
Let $M\subset\C$ be connected and let $\Phi\colon M\to\C^4$ be a
nowhere-vanishing holomorphic null curve whose associated immersion
$X=\re\int\Phi\,\dz$ is superminimal. Then, locally on $M$ and after a
constant rotation in $SO(4,\C)$, the null curve takes one of the two
forms
\[
  \Phi=\mu(1,\ii,g,-\ii g)
  \qquad\text{or}\qquad
  \Phi=\mu(1,\ii,g,\ii g),
\]
where $\mu$ is holomorphic and nowhere zero and $g$ is holomorphic.
Conversely, every null curve of either form defines a local superminimal
immersion. If $M$ is simply connected, or more generally if the real
periods of $\Phi\,\dz$ vanish, the immersion is global.
\end{proposition}

\begin{proof}
Restricting to an affine Segre chart, \Cref{eq:Phi-prime-g1-g2} gives
$g_1'g_2'\equiv 0$, hence by the identity principle $g_1'\equiv 0$ or
$g_2'\equiv 0$. A constant rotation in $SO(4,\C)$ moves the constant
component to the base point of its $\CP^1$-factor, yielding $g_2\equiv
0$ or $g_1\equiv 0$ in the Segre chart, which gives the two stated
forms.

Conversely, for either form a direct computation gives $\Phi^s=0$ and
$(\Phi')^s=0$, so \Cref{prop:superminimal-null-curve-condition} applies.
The immersion is global when $M$ is simply connected, or more generally
when the real periods of $\Phi\,\dz$ vanish.
\end{proof}

\begin{remark}
\label{rem:normal-form-classical}
This is the classical local normal form of superminimal null curves in
$\C^4$. The following sections rewrite the two ruling alternatives in
the $ALB$-language, expressing them as explicit conditions on the
Maurer--Cartan forms of the two spinorial factors.
\end{remark}
\section{Complex Quaternions and the \(ALB\)-Representation}
\label{sec:complex-quaternions}

In this section we recall the algebra of complex quaternions and use it
to express holomorphic null curves in $\C^4$ via factorizations
$\Phi=ALB$, where $A$ and $B$ are holomorphic maps with unit symmetrised
norm and $L$ is a holomorphic null factor. The multiplicativity of the
symmetrised norm is the basic algebraic mechanism behind the
representation. A variable null factor does not enlarge the local class
of null curves: on simply connected domains $L(z)$ can always be
absorbed into the spinorial factors, so the distinction between a moving
$L(z)$ and a fixed null element is a choice of gauge.

Quaternionic methods in conformal surface theory are developed in
\cite{BFLPP}; the complex-quaternionic approach to minimal surfaces
motivating the present work is \cite{AltavillaSchroeckerSirVrsek2025};
for quaternions in twistor descriptions of four-dimensional geometry
see also \cite{Bryant1982,AHS}.

\subsection{The algebra of complex quaternions}
\label{subsec:complex-quaternion-algebra}

Let $\HC:=\Hbb\otimes_\R\C$ be the algebra of complex quaternions, with
elements written as
\[
  q = q_0+q_1e_1+q_2e_2+q_3e_3, \qquad q_j\in\C,
\]
where $\ii$ commutes with the quaternionic units and
\[
  e_1^2=e_2^2=e_3^2=-1,
  \qquad
  e_1e_2=e_3,\quad e_2e_3=e_1,\quad e_3e_1=e_2.
\]
The \emph{quaternionic conjugate} of $q$ is
$q^c:=q_0-q_1e_1-q_2e_2-q_3e_3$, and the \emph{symmetrised norm} is
\begin{equation}
\label{eq:symmetrized-norm}
  q^s := qq^c = q^cq = q_0^2+q_1^2+q_2^2+q_3^2.
\end{equation}
Since the coefficients $q_j$ are complex, $q^s$ is a complex quadratic
form, not a positive-definite norm. We denote by $\re_\Hbb(q):=q_0$ the
scalar part of $q$; the space of \emph{pure complex quaternions} is
$\operatorname{Im}(\HC):=\{q\in\HC:q^c=-q\}=\C e_1\oplus\C e_2\oplus\C e_3$,
and every $q\in\HC$ splits as $q=q_0+\underline{q}$ with
$\underline{q}\in\operatorname{Im}(\HC)$.

The \emph{null cone} is $\NullCone:=\{q\in\HC:q^s=0\}$. The polar form
of $q\mapsto q^s$ is $\inner{q}{r}=q_0r_0+q_1r_1+q_2r_2+q_3r_3$, so
$\inner{q}{q}=q^s$. Under the identification
$q_0+q_1e_1+q_2e_2+q_3e_3\leftrightarrow(q_0,q_1,q_2,q_3)$ between
$\HC$ and $\C^4$, the symmetrised norm $q^s$ coincides with the
complex-bilinear square norm of \Cref{sec:preliminaries}, so the
notations $q^s$ and $\Phi^s$ are consistent.

The symmetrised norm is multiplicative:
\begin{equation}
\label{eq:symmetrized-norm-multiplicative}
  (qr)^s = q^sr^s, \qquad q,r\in\HC.
\end{equation}
Indeed,
$(qr)^s=qr(qr)^c=qrr^cq^c=q\,r^s\,q^c=r^sqq^c=q^sr^s$,
because $r^s\in\C$ is central.

We also use the $\C$-algebra isomorphism
$\Psi\colon\HC\to\operatorname{Mat}_{2\times 2}(\C)$ defined by
\begin{equation}
\label{eq:HC-matrix-model}
  \Psi(q_0+q_1e_1+q_2e_2+q_3e_3)
  :=
  \begin{pmatrix}
    q_0+\ii q_1 & q_2+\ii q_3\\
    -q_2+\ii q_3 & q_0-\ii q_1
  \end{pmatrix},
\end{equation}
which satisfies
\begin{equation}
\label{eq:det-Psi-symmetrized-norm}
  \det\Psi(q) = q^s.
\end{equation}
Consequently: $q\in\NullCone\setminus\{0\}$ if and only if $\Psi(q)$
has rank one; $q$ is invertible if and only if $q^s\neq 0$, with
$q^{-1}=q^c/q^s$; and $q^s=1$ if and only if $\Psi(q)\in SL(2,\C)$.
\subsection{The spin action and the null quadric}
\label{subsec:spin-action}

The group $\Sp(1,\C):=\{q\in\HC:q^s=1\}$, identified with $SL(2,\C)$
via $\Psi$, and its product $\Sp(1,\C)\times\Sp(1,\C)$ act on $\HC$
by two-sided multiplication $(A,B)\cdot q:=AqB$. By multiplicativity of
the symmetrised norm,
\begin{equation}
\label{eq:norm-multiplicativity}
  (AqB)^s = A^s\,q^s\,B^s.
\end{equation}
If $A^s=B^s=1$ the action preserves $q\mapsto q^s$ and hence the null
cone $\NullCone$. Replacing $B$ by $C^{-1}=C^c$, this recovers the
complex spin representation
$\Spin(4,\C)\simeq SL(2,\C)\times SL(2,\C)\to SO(4,\C)$.

Under $\Psi$, the projectivised null cone $Q^2$ is identified with the
variety of non-zero rank-one matrices in $\operatorname{Mat}_{2\times2}(\C)$.
Every such matrix writes as $uv^t$ with $u,v\in\C^2\setminus\{0\}$,
and the pair $([u],[v])\in\CP^1\times\CP^1$ determines its projective
class, recovering the Segre identification $Q^2\simeq\CP^1\times\CP^1$.
When $\Psi(q)=uv^t$ has rank one, left multiplication by $A\in\Sp(1,\C)$
acts on the column line $[u]$ and right multiplication by $B$ acts on
the row line $[v]$.

\begin{proposition}
\label{prop:transitivity-null-cone}
The action of $\Sp(1,\C)\times\Sp(1,\C)$ on $\NullCone\setminus\{0\}$
is transitive: for every two non-zero null elements $L,L_0\in\HC$
there exist $P,Q\in\Sp(1,\C)$ such that $L=PL_0Q$.
\end{proposition}

\begin{proof}
Via $\Psi$, it suffices to show that $SL(2,\C)\times SL(2,\C)$ acts
transitively on non-zero rank-one matrices. Write $M=uv^t$ and
$M_0=u_0v_0^t$ with $u,u_0,v,v_0\in\C^2\setminus\{0\}$. Since
$SL(2,\C)$ acts transitively on $\C^2\setminus\{0\}$, choose
$P,Q\in SL(2,\C)$ with $Pu_0=u$ and $v_0^tQ=v^t$. Then
\[
  PM_0Q = P(u_0v_0^t)Q = (Pu_0)(v_0^tQ) = uv^t = M.
\]
Transporting through $\Psi$ gives the result for $\Sp(1,\C)\times\Sp(1,\C)$.
\end{proof}

\subsection{The \(ALB\)-representation}
\label{subsec:ALB-representation}

\begin{definition}
\label{def:ALB-representation}
Let $M\subset\C$ be a domain. An \emph{$ALB$-representation} of a
holomorphic null curve $\Phi\colon M\to\HC$ is an expression $\Phi=ALB$,
where $A,B\colon M\to\Sp(1,\C)$ are holomorphic and
$L\colon M\to\HC\setminus\{0\}$ is a holomorphic null map, i.e.,
$L(z)^s=0$ for every $z\in M$. If $L$ is constant the representation
is called \emph{fixed-null}.
\end{definition}

\begin{proposition}
\label{prop:ALB-produces-null-curves}
Let $A,B\colon M\to\Sp(1,\C)$ be holomorphic and let
$L\colon M\to\HC\setminus\{0\}$ be holomorphic with $L^s=0$. Then
$\Phi:=ALB$ is a nowhere-vanishing holomorphic null curve, and
\[
  X = \re\int\Phi\,\dz
\]
defines a conformal minimal immersion $X\colon M\to\R^4$ on every simply
connected domain $M$.
\end{proposition}

\begin{proof}
The product $\Phi=ALB$ is holomorphic. By multiplicativity,
$\Phi^s=(ALB)^s=A^sL^sB^s=0$ since $L^s=0$. Since $A,B$ are
invertible and $L\neq 0$, we have $\Phi\neq 0$ everywhere. Thus $\Phi$
is a nowhere-vanishing holomorphic null curve, and the conclusion
follows from \Cref{subsec:holomorphic-null-curves}.
\end{proof}

Conversely, every nowhere-vanishing null curve has the tautological
representation $\Phi=1\cdot\Phi\cdot 1$, so the moving
$ALB$-representation is best understood as distributing the null curve
among two spinorial factors $A,B$ and a moving null factor $L$. The
next theorem shows that on simply connected domains the null factor can
always be absorbed into the spinorial factors.

\begin{theorem}[Fixed-null gauge]
\label{thm:global-ALB-representation}
Let $M\subset\C$ be a simply connected domain and let $\Phi\colon M\to\HC$
be holomorphic with $\Phi^s=0$ and $\Phi\neq 0$ everywhere. Fix a
non-zero null element $L_0\in\HC$. Then there exist holomorphic maps
$A,B\colon M\to\Sp(1,\C)$ such that $\Phi=AL_0B$.
\end{theorem}

\begin{proof}
We first treat the standard null element $L_*:=1+\ii e_1$. Under $\Psi$
this corresponds to $\Psi(L_*)=2E_{22}$, where $E_{22}$ has $1$ in the
$(2,2)$-entry and zeros elsewhere.

Set $F:=\Psi\circ\Phi$. Since $\Phi^s=0$ and $\Phi\neq 0$, the matrix
$F(z)$ has rank one for every $z\in M$, so its image
$\operatorname{image}F(z)\subset\C^2$ is a complex line varying
holomorphically in $z$. These lines define a holomorphic line subbundle
of the trivial bundle $M\times\C^2$. Since $M$ is a contractible Stein
Riemann surface, every holomorphic line bundle on $M$ is trivial, so
this subbundle admits a nowhere-vanishing holomorphic section
$u\colon M\to\C^2$ with $\operatorname{image}F(z)=\C u(z)$ for all
$z\in M$. Write $u=(u_1,u_2)^t$; since $u$ is nowhere zero, $u_1$ and
$u_2$ have no common zeros.

Since every column of $F(z)$ lies in $\C u(z)$, there exists a
holomorphic row vector $v^t=(v_1,v_2)$ with $F=uv^t$; $v$ is nowhere
zero because $F$ has rank one everywhere.

We now complete $u$ to an $SL(2,\C)$-frame. Since $M$ is Stein and
$u_1,u_2$ have no common zeros, the holomorphic Bézout equation is
solvable on $M$ \cite[Ch.~VII]{Forster1981}; see also
\cite[Ch.~VIII]{GunningRossi1965}. Hence there exist holomorphic $a,b$
with $au_2-bu_1=1$. Setting
\[
  \widetilde{A}
  :=
  \begin{pmatrix} a & u_1 \\ b & u_2 \end{pmatrix},
\]
we get $\det\widetilde{A}=au_2-bu_1=1$, so
$\widetilde{A}\colon M\to SL(2,\C)$.

Likewise, $v_1,v_2$ have no common zeros, so there exist holomorphic
$c,d$ with $cv_2-dv_1=2$. Setting
\[
  \widetilde{B}
  :=
  \begin{pmatrix} c & d \\ v_1/2 & v_2/2 \end{pmatrix},
\]
we get $\det\widetilde{B}=(cv_2-dv_1)/2=1$, so
$\widetilde{B}\colon M\to SL(2,\C)$.

By construction the second column of $\widetilde{A}$ is $u$ and the
second row of $\widetilde{B}$ is $v^t/2$, hence
\[
  \widetilde{A}\,(2E_{22})\,\widetilde{B} = uv^t = F.
\]
Transporting back through $\Psi$ gives holomorphic
$A_*,B_*\colon M\to\Sp(1,\C)$ with $\Phi=A_*L_*B_*$.

For general $L_0$, \Cref{prop:transitivity-null-cone} gives constant
$P,Q\in\Sp(1,\C)$ with $L_0=PL_*Q$, so $L_*=P^{-1}L_0Q^{-1}$.
Substituting,
\[
  \Phi = A_*L_*B_* = A_*P^{-1}L_0Q^{-1}B_*.
\]
Setting $A:=A_*P^{-1}$ and $B:=Q^{-1}B_*$ gives $\Phi=AL_0B$.
\end{proof}

\begin{corollary}[Gauge reduction of the moving null factor]
\label{cor:gauge-reduction-moving-L}
Let $M\subset\C$ be simply connected and let
$L\colon M\to\HC\setminus\{0\}$ be holomorphic with $L^s=0$. Fix a
non-zero null element $L_0\in\HC$. Then there exist holomorphic maps
$P,Q\colon M\to\Sp(1,\C)$ satisfying
\begin{equation}
\label{eq:moving-L-gauge-reduction}
  L = PL_0Q.
\end{equation}
Consequently, every moving $ALB$-representation $\Phi=ALB$ rewrites
in the fixed-null gauge as $\Phi=\widetilde{A}L_0\widetilde{B}$ with
$\widetilde{A}:=AP$ and $\widetilde{B}:=QB$.
\end{corollary}

\begin{proof}
Apply \Cref{thm:global-ALB-representation} to $L$ viewed as a
nowhere-vanishing null curve: this yields $P,Q\colon M\to\Sp(1,\C)$
with $L=PL_0Q$. Substituting into $\Phi=ALB$ gives
$\Phi=A(PL_0Q)B=(AP)L_0(QB)$.
\end{proof}

\begin{remark}
\label{rem:moving-versus-fixed-gauge}
The fixed-null gauge, and especially the canonical choice
$L_0=1+\ii e_1$, will be used in the sequel for explicit
computations; in that gauge the superminimality condition takes a
factored form in terms of the Maurer--Cartan forms of $A$ and $B$.
\end{remark}

\begin{remark}
\label{rem:ALB-domain-zeros}
The assumption $\Phi\neq 0$ is essential: since $A,B\in\Sp(1,\C)$
are invertible and $L\neq 0$, we have $ALB\neq 0$ automatically. To
allow branched surfaces one writes $\Phi=h\,ALB$ where $h$ is a
holomorphic scalar carrying the zero divisor. Away from zeros of $h$
the projective null curve $[h\,ALB]=[ALB]$ is unchanged, so branching
does not affect the ruling condition for superminimality; the scalar
$h$ carries the branch divisor.
\end{remark}

\subsection{Minimal surfaces in \texorpdfstring{$\R^4$}{R4}
  via the \texorpdfstring{$ALB$}{ALB}-representation}
\label{subsec:main-theorem-R4}

We now translate the Weierstrass representation of conformal minimal
immersions into the $ALB$-language in two equivalent forms: a
fixed-null gauge, in which $L$ is a prescribed constant null element,
and a moving-null form, in which $L=L(z)$ is allowed to vary.

\begin{theorem}[$ALB$-representation of minimal surfaces in $\R^4$]
\label{thm:main-R4}
Let $M\subset\C$ be a simply connected domain. A smooth map
$X\colon M\to\R^4$ is a conformal minimal immersion if and only if,
for any fixed non-zero null element $L_0\in\HC$, there exist
holomorphic maps $A,B\colon M\to\Sp(1,\C)$ and a constant $c\in\R^4$
such that
\begin{equation}
\label{eq:main-formula-R4-fixed}
  X(z) = c+\re\int_{z_0}^{z}A(\zeta)\,L_0\,B(\zeta)\,\dd\zeta.
\end{equation}
Equivalently, $X$ is a conformal minimal immersion if and only if
there exist holomorphic $A,B\colon M\to\Sp(1,\C)$, a holomorphic null
map $L\colon M\to\HC\setminus\{0\}$ with $L(z)^s=0$, and $c\in\R^4$
such that
\begin{equation}
\label{eq:main-formula-R4-moving}
  X(z) = c+\re\int_{z_0}^{z}A(\zeta)\,L(\zeta)\,B(\zeta)\,\dd\zeta.
\end{equation}
\end{theorem}

\begin{proof}
Sufficiency of \Cref{eq:main-formula-R4-moving} follows from
\Cref{prop:ALB-produces-null-curves}: $\Phi:=ALB$ is a
nowhere-vanishing holomorphic null curve, and $X=c+\re\int\Phi\,\dz$
is a conformal minimal immersion. The fixed-null formula
\Cref{eq:main-formula-R4-fixed} is the special case
$L(\zeta)\equiv L_0$.

Conversely, let $X\colon M\to\R^4$ be a conformal minimal immersion
and set $\Phi:=2X_z$. By \Cref{subsec:holomorphic-null-curves}, $\Phi$
is a nowhere-vanishing holomorphic null curve. Fix $L_0\in\HC$ with
$L_0^s=0$ and $L_0\neq 0$. By \Cref{thm:global-ALB-representation},
there exist holomorphic $A,B\colon M\to\Sp(1,\C)$ with $\Phi=AL_0B$.
Since $2X_z=\Phi$, the difference $X-\re\int\Phi\,\dz$ is constant,
so
\[
  X(z) = c+\re\int_{z_0}^{z}A(\zeta)\,L_0\,B(\zeta)\,\dd\zeta
\]
for some $c\in\R^4$. The moving-null form follows by taking
$L\equiv L_0$, or more generally by \Cref{cor:gauge-reduction-moving-L}.
\end{proof}

\begin{remark}
\label{rem:immersion-condition-ALB}
The immersion condition is automatic from the $ALB$-data: since
$A^s=B^s=1$ and $L\neq 0$ we have $\Phi=ALB\neq 0$, hence
$\inner{\Phi}{\overline{\Phi}}>0$ everywhere. If an additional scalar
$h$ is included, as in $\Phi=h\,ALB$, then branch points occur
precisely at the zeros of $h$.
\end{remark}

\begin{remark}
\label{rem:R3-special-case}
When $\Phi$ takes values in $\operatorname{Im}(\HC)$, the formula
$X=c+\re\int\Phi\,\dz$ produces a conformal minimal immersion into
$\R^3\subset\R^4$. For example, take $L':=e_1+\ii e_2$, which
satisfies $(L')^s=0$. For holomorphic $A\colon M\to\Sp(1,\C)$, the
product $\Phi=AL'A^c$ is pure: for $k=1,2$ one has
$(Ae_kA^c)^c=Ae_k^cA^c=-Ae_kA^c$,
so $\Phi^c=-\Phi$. Hence
\[
  X = c+\re\int AL'A^c\,\dz
\]
takes values in $\R^3\subset\R^4$. This recovers the
complex-quaternionic representation of minimal surfaces in $\R^3$
developed in \cite{AltavillaSchroeckerSirVrsek2025}, where the
symmetry condition $B=A^c$ is established in full.
\end{remark}

\section{The Superminimality Condition in the \(ALB\)-Representation}
\label{sec:superminimal-ALB}

We now translate the superminimality condition $(\Phi')^s=0$ into the
$ALB$-language. In the moving-null representation $\Phi=ALB$ with
$\alpha=A^cA'$ and $\beta=B'B^c$, the derivative of the null factor
contributes an additional term; the superminimality condition becomes
\[
  (\alpha L+L'+L\beta)^s=0.
\]
In the fixed-null gauge ($L$ constant, $L'=0$) this reduces to
\[
  (\alpha L+L\beta)^s=0.
\]

\subsection{Logarithmic derivatives and the derivative of
  \texorpdfstring{$\Phi$}{Phi}}
\label{subsec:logarithmic-derivatives}

Let $\Phi=ALB$ be an $ALB$-representation. Since $A,B\in\Sp(1,\C)$
are invertible, with $A^{-1}=A^c$ and $B^{-1}=B^c$, we define the
left and right Maurer--Cartan forms
\begin{equation}
\label{eq:logarithmic-derivatives}
  \alpha := A^cA', \qquad \beta := B'B^c,
\end{equation}
so that $A'=A\alpha$ and $B'=\beta B$. Differentiating $A^cA=1$ and
$BB^c=1$ gives $\alpha^c=-\alpha$ and $\beta^c=-\beta$, so $\alpha$
and $\beta$ take values in $\operatorname{Im}(\HC)$; we write
$\alpha=\alpha_1e_1+\alpha_2e_2+\alpha_3e_3$ and
$\beta=\beta_1e_1+\beta_2e_2+\beta_3e_3$.

Differentiating $\Phi=ALB$ gives
\[
  \Phi' = A'LB+AL'B+ALB'
        = A\alpha LB+AL'B+AL\beta B
        = A(\alpha L+L'+L\beta)B.
\]
By multiplicativity of the symmetrised norm and $A^s=B^s=1$,
\begin{equation}
\label{eq:Phi-prime-norm-moving-ALB}
  (\Phi')^s = (\alpha L+L'+L\beta)^s.
\end{equation}

\begin{theorem}[Superminimality in moving $ALB$-form]
\label{thm:superminimal-moving-ALB}
Let $\Phi=ALB$ with $A,B\colon M\to\Sp(1,\C)$ holomorphic and
$L\colon M\to\HC\setminus\{0\}$ holomorphic with $L^s=0$. Let
$\alpha=A^cA'$ and $\beta=B'B^c$. Then $X=\re\int\Phi\,\dz$ is
superminimal if and only if
\begin{equation}
\label{eq:superminimal-moving-ALB}
  (\alpha L+L'+L\beta)^s=0.
\end{equation}
\end{theorem}

\begin{proof}
By \Cref{prop:superminimal-null-curve-condition}, $X$ is superminimal
if and only if $(\Phi')^s=0$. By \Cref{eq:Phi-prime-norm-moving-ALB},
this is equivalent to \Cref{eq:superminimal-moving-ALB}.
\end{proof}

\begin{remark}
\label{rem:moving-L-term}
The term $L'$ records the infinitesimal motion of the null direction
of $L(z)$. Since $L^s=0$, differentiating gives $\inner{L'}{L}=0$,
so $L'$ is tangent to the null cone along $L$; however, $(L')^s$ need
not vanish. Thus \Cref{eq:superminimal-moving-ALB} does not factor
into a product involving only $A$ and $B$. The factorization appears
in the fixed-null gauge, where $L'=0$.
\end{remark}

\begin{corollary}[Fixed-null gauge]
\label{cor:superminimal-fixed-null-gauge}
Assume $L\in\HC\setminus\{0\}$ is constant and null. Let $\Phi=ALB$
with $A,B\colon M\to\Sp(1,\C)$, and let $\alpha=A^cA'$,
$\beta=B'B^c$. Then $X=\re\int\Phi\,\dz$ is superminimal if and only
if $(\alpha L+L\beta)^s=0$, equivalently $\inner{\alpha L}{L\beta}=0$.
\end{corollary}

\begin{proof}
Since $L$ is constant, $L'=0$, so \Cref{eq:superminimal-moving-ALB}
reduces to $(\alpha L+L\beta)^s=0$. Using $(u+v)^s=u^s+v^s+2\inner{u}{v}$
with $u=\alpha L$ and $v=L\beta$, and noting that
$(\alpha L)^s=\alpha^sL^s=0$ and $(L\beta)^s=L^s\beta^s=0$ since
$L^s=0$, we obtain
\[
  (\alpha L+L\beta)^s = 2\inner{\alpha L}{L\beta}.
\]
Hence $(\alpha L+L\beta)^s=0$ if and only if $\inner{\alpha L}{L\beta}=0$.
\end{proof}

\begin{remark}
\label{rem:geometric-meaning-fixed-null}
In the fixed-null gauge, $\alpha L$ and $L\beta$ describe the two
infinitesimal spinorial motions of the projective null curve
$[\Phi]\colon M\to Q^2\simeq\CP^1\times\CP^1$. The condition
$\inner{\alpha L}{L\beta}=0$ asserts that their sum is null, or
equivalently that the tangent direction of $[\Phi]$ lies in one of
the two rulings of the quadric.
\end{remark}

\subsection{Normal form in a fixed-null gauge}
\label{subsec:normal-form-null-vector}

We now specialise the fixed-null gauge to the standard null element
$L=1+\ii e_1$. With $\alpha$ and $\beta$ as in
\Cref{eq:logarithmic-derivatives}, a direct computation gives
\[
  \alpha L
  =
  -\ii\alpha_1+\alpha_1e_1
  +(\alpha_2+\ii\alpha_3)e_2
  +(\alpha_3-\ii\alpha_2)e_3,\quad
  L\beta
  =
  -\ii\beta_1+\beta_1e_1
  +(\beta_2-\ii\beta_3)e_2
  +(\beta_3+\ii\beta_2)e_3.
\]
Setting $U:=\alpha_2+\ii\alpha_3$, $V:=\beta_2-\ii\beta_3$, and
$W:=\alpha_1+\beta_1$, we obtain
$\alpha L+L\beta=-\ii W+We_1+(U+V)e_2+\ii(V-U)e_3$.
Since $(-\ii W)^2+W^2=0$, the scalar and $e_1$-components contribute
zero, and
\begin{equation}
\label{eq:normal-form-product-components}
  (\alpha L+L\beta)^s
  =
  (U+V)^2+\bigl(\ii(V-U)\bigr)^2
  =
  4UV
  =
  4(\alpha_2+\ii\alpha_3)(\beta_2-\ii\beta_3).
\end{equation}

\begin{theorem}[Superminimality in the normal form]
\label{thm:superminimal-ALB-normal-form}
Assume $L=1+\ii e_1$. Let $X=\re\int ALB\,\dz$ be a conformal minimal
immersion with holomorphic $A,B\colon M\to\Sp(1,\C)$, and let
$\alpha=A^cA'$, $\beta=B'B^c$. Then $X$ is superminimal if and only if
\begin{equation}
\label{eq:superminimal-ALB-product}
  (\alpha_2+\ii\alpha_3)(\beta_2-\ii\beta_3)=0.
\end{equation}
If $M$ is connected, this is equivalent to at least one of
\begin{equation}
\label{eq:superminimal-ALB-alternatives}
  \alpha_2+\ii\alpha_3\equiv0,
  \qquad
  \beta_2-\ii\beta_3\equiv0.
\end{equation}
For non-planar superminimal immersions, exactly one alternative holds.
\end{theorem}

\begin{proof}
Since $L$ is constant, \Cref{cor:superminimal-fixed-null-gauge} gives
$(\alpha L+L\beta)^s=0$ as the superminimality condition. By
\Cref{eq:normal-form-product-components}, this is equivalent to
\Cref{eq:superminimal-ALB-product}.

If $M$ is connected, the functions $\alpha_2+\ii\alpha_3$ and
$\beta_2-\ii\beta_3$ are holomorphic; if their product vanishes
identically, the identity principle forces at least one to vanish
identically. If both vanish, then by \Cref{prop:AL-LB-geometric-meaning}
below, both $[AL]$ and $[LB]$ are constant, so
$[\Phi]=[ALB]\colon M\to Q^2$ is constant. Hence there exist a fixed
non-zero null vector $\Phi_0\in\C^4$ and a nowhere-vanishing
holomorphic function $\mu$ with $\Phi=\mu\Phi_0$. The image of
$X=\re\int\mu(z)\Phi_0\,\dz$ lies in the real two-plane spanned by
$\re\Phi_0$ and $\im\Phi_0$, so the surface is planar. Therefore, in
the non-planar case, exactly one alternative holds.
\end{proof}

\begin{corollary}
\label{cor:superminimal-AB-components}
Let $L=1+\ii e_1$, and write $A=A_0+A_1e_1+A_2e_2+A_3e_3$ and
$B=B_0+B_1e_1+B_2e_2+B_3e_3$. Then $X=\re\int ALB\,\dz$ is
superminimal if and only if
\begin{equation}
\label{eq:superminimal-AB-components}
\begin{aligned}
  &\bigl[(A_0-\ii A_1)(A_2+\ii A_3)'-(A_0-\ii A_1)'(A_2+\ii A_3)\bigr]\\
  &\quad\cdot\bigl[(B_0-\ii B_1)(B_2-\ii B_3)'
    -(B_0-\ii B_1)'(B_2-\ii B_3)\bigr]=0.
\end{aligned}
\end{equation}
On a connected $M$, this is equivalent to at least one of the
holomorphic maps $[A_0-\ii A_1:A_2+\ii A_3]\colon M\to\CP^1$ and
$[B_0-\ii B_1:B_2-\ii B_3]\colon M\to\CP^1$ being constant.
\end{corollary}

\begin{proof}
We start by computing $\alpha_2+\ii\alpha_3$
Writing
$$
  A^c=A_0-A_1e_1-A_2e_2-A_3e_3,
  \qquad
  A'=A_0'+A_1'e_1+A_2'e_2+A_3'e_3,
$$
and expanding the product $A^cA'$, the $e_2$- and $e_3$-components are
$$
  \alpha_2
  =
  A_0A_2'-A_2A_0'+A_1A_3'-A_3A_1',\quad
  \alpha_3
  =
  A_0A_3'-A_3A_0'-A_1A_2'+A_2A_1'.
$$
Hence
$$
\begin{aligned}
  \alpha_2+\ii\alpha_3
  &=
  A_0A_2'-A_2A_0'+A_1A_3'-A_3A_1' \\
  &\qquad
  +\ii\bigl(A_0A_3'-A_3A_0'-A_1A_2'+A_2A_1'\bigr) \\
  &=
  (A_0-\ii A_1)(A_2+\ii A_3)'
  -(A_0-\ii A_1)'(A_2+\ii A_3).
\end{aligned}
$$
Analogously, from $\beta=B'B^c$, we get
\[
  \beta_2-\ii\beta_3
  =
  (B_0-\ii B_1)(B_2-\ii B_3)'-(B_0-\ii B_1)'(B_2-\ii B_3).
\]
Substituting into \Cref{eq:superminimal-ALB-product} gives
\Cref{eq:superminimal-AB-components}.

Since $A^s=B^s=1$, the matrix model shows that the pairs
$(A_0-\ii A_1,A_2+\ii A_3)$ and $(B_0-\ii B_1,B_2-\ii B_3)$ have no
common zeros, so they define holomorphic maps to $\CP^1$. A pair of
holomorphic functions $(f,g)$ without common zeros defines a constant
map $[f:g]\colon M\to\CP^1$ if and only if its Wronskian $fg'-f'g$
vanishes identically. Hence the vanishing of at least one factor in
\Cref{eq:superminimal-AB-components} is equivalent to the constancy of
at least one of the two projective maps.
\end{proof}

\begin{proposition}
\label{prop:AL-LB-geometric-meaning}
Let $L=1+\ii e_1$. Then $\alpha_2+\ii\alpha_3\equiv0$ if and only if
$[AL]\colon M\to\CP^3$ is constant. Similarly,
$\beta_2-\ii\beta_3\equiv0$ if and only if $[LB]\colon M\to\CP^3$ is
constant.
\end{proposition}

\begin{proof}
Since $A'=A\alpha$, we have $(AL)'=A(\alpha L)$. Since $A$ is
invertible, $[AL]$ is constant if and only if
\[
  \alpha L\in\C L.
\]
From the expansion of $\alpha L$ computed above, the first two terms
satisfy $-\ii\alpha_1+\alpha_1e_1=(-\ii\alpha_1)L$, so $\alpha L\in\C L$
if and only if the $e_2$ and $e_3$ components vanish, i.e.,
$\alpha_2+\ii\alpha_3=0$.

The proof for $B$ is analogous. Since $B'=\beta B$, we have
$(LB)'=L\beta B$. Since $B$ is invertible, $[LB]$ is constant if and
only if $L\beta\in\C L$. From the expansion of $L\beta$ computed above,
$L\beta\in\C L$ if and only if $\beta_2-\ii\beta_3=0$.
\end{proof}

\begin{corollary}
\label{cor:superminimal-rulings}
Let $M$ be connected and let $X=\re\int ALB\,\dz$ be a conformal
minimal immersion in the fixed-null gauge $L=1+\ii e_1$. Then $X$ is
superminimal if and only if one of the two spinorial components of the
generalised Gauss map $[\Phi]\colon M\to Q^2\simeq\CP^1\times\CP^1$
is constant.
\end{corollary}

\begin{proof}
By \Cref{thm:superminimal-ALB-normal-form}, $X$ is superminimal iff
$\alpha_2+\ii\alpha_3\equiv0$ or $\beta_2-\ii\beta_3\equiv0$. By
\Cref{prop:AL-LB-geometric-meaning}, these are equivalent to the
constancy of $[AL]$ and $[LB]$ respectively.

In the matrix model, $\Psi(L)=2E_{22}$, so
\[
  \Psi(ALB) = \widetilde{A}\,(2E_{22})\,\widetilde{B},
\]
which is twice the outer product of the second column of $\widetilde{A}$
and the second row of $\widetilde{B}$. Hence the two components of the
Segre map $[\Phi]\colon M\to Q^2\simeq\CP^1\times\CP^1$ are encoded by
$[AL]$ and $[LB]$; their constancy is precisely the constancy of one
spinorial component of the generalised Gauss map.
\end{proof}

\begin{remark}
\label{rem:ALB-rulings-geometric-summary}
\Cref{cor:superminimal-rulings} is the quaternionic translation of the
classical criterion \Cref{eq:superminimal-gauss}. The $ALB$-representation
makes this ruling condition explicit in terms of the Maurer--Cartan forms
$\alpha$ and $\beta$ via the factored formula
\Cref{eq:normal-form-product-components}.
\end{remark}

\begin{remark}[Intrinsic formulation]
\label{rem:intrinsic-pL}
The conditions $\alpha_2+\ii\alpha_3=0$ and $\beta_2-\ii\beta_3=0$
depend on the choice $L=1+\ii e_1$, but their vanishing has an
invariant meaning in any fixed-null gauge. For a fixed non-zero null
element $L\in\HC$, define the infinitesimal projective stabilisers
\begin{equation}
\label{eq:pL-definitions}
  \mathfrak{p}_L^\ell
  := \{\gamma\in\operatorname{Im}(\HC):\gamma L\in\C L\},
  \qquad
  \mathfrak{p}_L^r
  := \{\gamma\in\operatorname{Im}(\HC):L\gamma\in\C L\},
\end{equation}
the Lie algebras of the projective stabilisers of $[L]$ under left and
right multiplication. In the normal form $L=1+\ii e_1$,
{
let
$$
  \gamma=a e_1+b e_2+c e_3
  \in \operatorname{Im}(\HC).
$$
Then
$$
\begin{aligned}
  \gamma L
  &=
  (a e_1+b e_2+c e_3)(1+\ii e_1) \\
  &=
  a e_1+b e_2+c e_3
  +\ii a e_1^2+\ii b e_2e_1+\ii c e_3e_1 \\
  &=
  -\ii a
  +a e_1
  +(b+\ii c)e_2
  +(c-\ii b)e_3.
\end{aligned}
$$
Thus $\gamma L\in\C L$ if and only if the $e_2$- and
$e_3$-components vanish, that is,
$$
  b+\ii c=0,
  \qquad
  c-\ii b=0.
$$
Equivalently, $c=\ii b$. Hence
$$
  \gamma
  =
  a e_1+b(e_2+\ii e_3)
  =
  a e_1-\ii b(e_3-\ii e_2),
$$
and therefore
$$
  \mathfrak{p}_L^\ell
  =
  \C e_1\oplus\C(e_3-\ii e_2).
$$

Similarly,
$$
\begin{aligned}
  L\gamma
  &=
  (1+\ii e_1)(a e_1+b e_2+c e_3) \\
  &=
  a e_1+b e_2+c e_3
  +\ii a e_1^2+\ii b e_1e_2+\ii c e_1e_3 \\
  &=
  -\ii a
  +a e_1
  +(b-\ii c)e_2
  +(c+\ii b)e_3.
\end{aligned}
$$
Thus $L\gamma\in\C L$ if and only if
$$
  b-\ii c=0,
  \qquad
  c+\ii b=0,
$$
equivalently $c=-\ii b$. Hence
$$
  \gamma
  =
  a e_1+b(e_2-\ii e_3)
  =
  a e_1+\ii b(e_3+\ii e_2),
$$
and therefore
$$
  \mathfrak{p}_L^r
  =
  \C e_1\oplus\C(e_3+\ii e_2).
$$
}
Thus $\alpha_2+\ii\alpha_3=0$ iff $\alpha\in\mathfrak{p}_L^\ell$, and
$\beta_2-\ii\beta_3=0$ iff $\beta\in\mathfrak{p}_L^r$. In an arbitrary
fixed-null gauge, the superminimality dichotomy reads intrinsically as
\[
  \alpha\in\mathfrak{p}_L^\ell
  \qquad\text{or}\qquad
  \beta\in\mathfrak{p}_L^r.
\]
Setting $\eta_+:=e_2+\ii e_3$ and $\eta_-:=e_2-\ii e_3$, the
membership conditions are equivalently
$\alpha\in\mathfrak{p}_L^\ell\Leftrightarrow\inner{\alpha}{\eta_+}=0$
and
$\beta\in\mathfrak{p}_L^r\Leftrightarrow\inner{\beta}{\eta_-}=0$,
where orthogonality is complex-bilinear.
\end{remark}

\begin{remark}[Global constancy of the complex structure]
\label{rem:global-constant-J}
For a non-planar superminimal surface on a connected domain, exactly
one component of the generalised Gauss map is constant. This component
determines a fixed orthogonal complex structure
$J\in J^\pm(\R^4)\simeq S^2$ such that $X$ is globally a holomorphic
curve in $(\R^4,J)$; see \cite{Bryant1982,EellsSalamon1985}. In the
fixed $ALB$-gauge $L=1+\ii e_1$, the alternative
$\alpha_2+\ii\alpha_3\equiv0$ corresponds to one spin orientation and
$\beta_2-\ii\beta_3\equiv0$ to the other. Non-constant twistor lifts
arise in the broader theory of superconformal, not necessarily minimal,
surfaces.
\end{remark}

\subsection{ODE parametrization via the projective stabilizer}
\label{subsec:ODE-parametrization}

We now turn the fixed-gauge superminimality criterion into a
constructive description of the $ALB$-data, working in the gauge
$L=1+\ii e_1$. Set $N_-:=e_3-\ii e_2$ and $N_+:=e_3+\ii e_2$;
by \Cref{rem:intrinsic-pL},
\[
  \mathfrak{p}_L^\ell=\C e_1\oplus\C N_-,
  \qquad
  \mathfrak{p}_L^r=\C e_1\oplus\C N_+,
\]
and one verifies $N_-L=0$, $LN_+=0$, $N_-^2=N_+^2=0$. The left
superminimal branch is $\alpha=fe_1+hN_-$ and the right is
$\beta=ge_1+kN_+$, for holomorphic functions $f,h$ and $g,k$
respectively. These equations can be integrated explicitly: the
generators satisfy
\[
  e_1N_-=-\ii N_-,\quad N_-e_1=\ii N_-,\quad [e_1,N_-]=-2\ii N_-,
\]
and analogously $[e_1,N_+]=2\ii N_+$, so $\C e_1\oplus\C N_\pm$
are solvable two-dimensional Lie subalgebras of $\operatorname{Im}(\HC)$.
First-order systems of this type also appear in the twistor description
of superminimal surfaces in spheres; compare
\cite{Bryant1982,ChiFernandezWu1999}.

\begin{proposition}[ODE parametrization in the fixed-null gauge]
\label{prop:ODE-parametrization}
Let $M\subset\C$ be simply connected and fix $L=1+\ii e_1$.

For the left branch, choose holomorphic $f,h\colon M\to\C$, an
initial value $A_0\in\Sp(1,\C)$, and an arbitrary holomorphic
$B\colon M\to\Sp(1,\C)$. Let $F$ be the primitive of $f$ with
$F(z_0)=0$, and set
\[
  K(z) := e^{2\ii F(z)}\int_{z_0}^{z}h(\zeta)\,e^{-2\ii F(\zeta)}\,\dd\zeta.
\]
Then
\begin{equation}
\label{eq:A-explicit-left-ODE}
  A(z) = A_0\bigl(\cos F(z)+\sin F(z)\,e_1\bigr)\bigl(1+K(z)N_-\bigr)
\end{equation}
is the unique solution of
\begin{equation}
\label{eq:ODE-A}
  A' = A(fe_1+hN_-),\qquad A(z_0)=A_0,
\end{equation}
and satisfies $A^s\equiv1$. Hence $\Phi=ALB$ is a nowhere-vanishing
holomorphic null curve and $X=\re\int\Phi\,\dz$ is superminimal.

For the right branch, choose holomorphic $g,k\colon M\to\C$, an
initial value $B_0\in\Sp(1,\C)$, and an arbitrary holomorphic
$A\colon M\to\Sp(1,\C)$. Let $G$ be the primitive of $g$ with
$G(z_0)=0$, and set
\[
  R(z) := e^{2\ii G(z)}\int_{z_0}^{z}k(\zeta)\,e^{-2\ii G(\zeta)}\,\dd\zeta.
\]
Then
\begin{equation}
\label{eq:B-explicit-right-ODE}
  B(z) = \bigl(1+R(z)N_+\bigr)\bigl(\cos G(z)+\sin G(z)\,e_1\bigr)B_0
\end{equation}
is the unique solution of
\begin{equation}
\label{eq:ODE-B}
  B' = (ge_1+kN_+)B,\qquad B(z_0)=B_0,
\end{equation}
and satisfies $B^s\equiv1$; the corresponding $ALB$-data define a
superminimal conformal minimal immersion.

Conversely, every fixed-gauge superminimal $ALB$-representation
satisfying $\alpha\in\mathfrak{p}_L^\ell$ (resp.\ $\beta\in\mathfrak{p}_L^r$)
is locally, and globally on simply connected domains, obtained from
the left (resp.\ right) branch above.
\end{proposition}

\begin{proof}
We prove the left branch; the right is analogous.

Let $E:=e_1$ and $N:=N_-=e_3-\ii e_2$. Then
\[
  E^2=-1,\qquad N^2=0,\qquad EN=-\ii N,\qquad NE=\ii N.
\]
Set $R_F:=\cos F+\sin F\,E$ and $U_K:=1+KN$. Then
\[
  R_F'=R_FfE,\qquad U_K'=K'N,\qquad U_K^{-1}=1-KN.
\]
{
We have that,
$$
\begin{aligned}
  U_K^{-1}EU_K
  &=
  (1-KN)E(1+KN) \\
  &=
  E+KEN-KNE-K^2NEN.
\end{aligned}
$$
Since
$$
  EN=-\ii N,
  \qquad
  NE=\ii N,
  \qquad
  N^2=0,
$$
we also have
$$
  NEN=N(EN)=-\ii N^2=0.
$$
Therefore
$$
  U_K^{-1}EU_K
  =
  E+EKN-KNE-KNEKN
  =
  E-\ii KN-\ii KN
  =
  E-2\ii KN.
$$
Moreover,
$$
  U_K^{-1}U_K'
  =
  (1-KN)K'N
  =
  K'N-KK'N^2
  =
  K'N.
$$
Consequently,
$$
\begin{aligned}
  A^{-1}A'
  &=
  (A_o R_F U_K)^{-1}(A_0 R_F fE U_K+A_o R_F U_K')\\
  &=U_K^{-1}R_F^{-1} A_0^{-1}A_0 R_F fE U_K+U_K^{-1}R_F^{-1}A_0^{-1}A_0 R_F U_K'\\
  &=
  U_K^{-1}(fE)U_K+U_K^{-1}U_K' \\
  &=
  f(E-2\ii KN)+K'N \\
  &=
  fE+(K'-2\ii fK)N.
\end{aligned}
$$
}

From the definition of $K$, we have $K'-2\ii fK=h$. Thus
$A^{-1}A'=fE+hN$, equivalently $A'=A(fe_1+hN_-)$. The initial
conditions $F(z_0)=0$ and $K(z_0)=0$ give $A(z_0)=A_0$.

It remains to verify $A^s\equiv1$. Since $E^c=-E$, $N^c=-N$, and
$N^2=0$,
\[
  R_F^s=1,\qquad U_K^s=(1+KN)(1-KN)=1.
\]
Hence $A^s=A_0^sR_F^sU_K^s=1$, so $A\colon M\to\Sp(1,\C)$.

For any holomorphic $B\colon M\to\Sp(1,\C)$, set $\Phi=ALB$. Then
$\Phi$ is a nowhere-vanishing null curve, and since
\[
  \alpha=A^cA'=fe_1+hN_-\in\mathfrak{p}_L^\ell,
\]
we have $\alpha_2+\ii\alpha_3\equiv0$, so $X=\re\int\Phi\,\dz$ is
superminimal by \Cref{thm:superminimal-ALB-normal-form}.

Conversely, if $\alpha\in\mathfrak{p}_L^\ell$, then $\alpha=fe_1+hN_-$
for holomorphic $f,h$. Since $A'=A\alpha$, the factor $A$ is the
solution of \Cref{eq:ODE-A} with initial value $A(z_0)$, and the
explicit formula \Cref{eq:A-explicit-left-ODE} follows by uniqueness.
{
Indeed, let $\widetilde{A}$ be any other holomorphic solution of
$$
  \widetilde{A}'=\widetilde{A}(f e_1+hN_-),
  \qquad
  \widetilde{A}(z_0)=A_0.
$$
Since both $A$ and $\widetilde{A}$ take values in $\Sp(1,\C)$, they are
pointwise invertible. Set
$$
  M:=\widetilde{A}A^{-1}.
$$
Differentiating $AA^{-1}=1$ and using $A'=A(f e_1+hN_-)$, we obtain
$$
  (A^{-1})'=-(f e_1+hN_-)A^{-1}.
$$
Therefore
\begin{align*}
  M'
  &=
  \widetilde{A}'A^{-1}+\widetilde{A}(A^{-1})' \\
  &=
  \widetilde{A}(f e_1+hN_-)A^{-1}
  -\widetilde{A}(f e_1+hN_-)A^{-1} \\
  &=
  0.
\end{align*}
Hence $M$ is constant on $M$. Evaluating at $z_0$ gives
$$
  M(z_0)
  =
  \widetilde{A}(z_0)A(z_0)^{-1}
  =
  A_0A_0^{-1}
  =
  1,
$$
so $M\equiv 1$ and therefore $\widetilde{A}\equiv A$. This proves the
uniqueness of the solution with the prescribed initial value.

The simply connectedness of $M$ is used only to ensure that the
primitive
$$
  F(z)=\int_{z_0}^{z}f(\zeta)\,\dd\zeta
$$
and the integral defining
$$
  K(z)
  =
  e^{2\ii F(z)}
  \int_{z_0}^{z}h(\zeta)e^{-2\ii F(\zeta)}\,\dd\zeta
$$
are globally single-valued. Thus the explicit local solution given by
\eqref{eq:A-explicit-left-ODE} extends to a globally defined holomorphic
solution on all of $M$.
}
The right branch uses $N_+=e_3+\ii e_2$ with $e_1N_+=\ii N_+$ and
$N_+e_1=-\ii N_+$, and solves $B'=(ge_1+kN_+)B$.
\end{proof}

\begin{remark}
\label{rem:ODE-gauge-meaning}
The explicit formula \Cref{eq:A-explicit-left-ODE} separates the exact
stabiliser freedom from the projective stabiliser freedom. The
unipotent factors $1+KN_-$ and $1+RN_+$ belong to the exact
stabilisers of $L$ (since $N_-L=LN_+=0$): they may make $A$ or $B$
non-constant while leaving $AL$ or $LB$ unchanged. By contrast, the
factor $R_F=\cos F+\sin F\,e_1$ acts on $L$ by scalar multiplication,
\[
  R_FL=e^{-\ii F}L,\qquad LR_F=e^{-\ii F}L,
\]
preserving the projective null line $[L]$ but not $L$ itself. Thus the
$N_\mp$-terms are exact stabiliser directions, while the $e_1$-term is
a projective stabiliser direction.
\end{remark}

\subsection{The associate family}
\label{subsec:associate-family}

Given a conformal minimal immersion $X=c+\re\int\Phi\,\dz$, its
\emph{associate family} is the one-parameter deformation
\begin{equation}
\label{eq:associate-family}
  X_\theta := c+\re\!\int e^{\ii\theta}\Phi\,\dz,
  \qquad \theta\in[0,2\pi).
\end{equation}
Each $X_\theta$ is again a conformal minimal immersion. In the
$ALB$-representation, the associate family is obtained by multiplying
the null factor by the constant scalar $e^{\ii\theta}$.

\begin{proposition}[Associate family in $ALB$-language]
\label{prop:associate-family}
Let $X=c+\re\int ALB\,\dz$ be a conformal minimal immersion. For
$\theta\in[0,2\pi)$, set $L_\theta:=e^{\ii\theta}L$. Then the
associate family of $X$ is
\begin{equation}
\label{eq:associate-ALB}
  X_\theta = c+\re\!\int A\,L_\theta\,B\,\dz.
\end{equation}
Moreover, setting $\Phi_\theta:=AL_\theta B=e^{\ii\theta}\Phi$,
\begin{equation}
\label{eq:associate-supermin}
  (\Phi_\theta')^s = e^{2\ii\theta}(\Phi')^s.
\end{equation}
In particular, if $X$ is superminimal then every $X_\theta$ is
superminimal.
\end{proposition}

\begin{proof}
Since $L_\theta=e^{\ii\theta}L$, we have $L_\theta^s=e^{2\ii\theta}L^s=0$,
so $L_\theta$ is a null factor. Moreover,
$\Phi_\theta=AL_\theta B=e^{\ii\theta}ALB=e^{\ii\theta}\Phi$,
which gives \Cref{eq:associate-ALB}. Differentiating,
$\Phi_\theta'=e^{\ii\theta}\Phi'$ (since $e^{\ii\theta}$ is constant),
so $(\Phi_\theta')^s=e^{2\ii\theta}(\Phi')^s$. If $(\Phi')^s=0$ then
$(\Phi_\theta')^s=0$ for every $\theta$, so superminimality is
preserved.
\end{proof}

\begin{remark}
\label{rem:associate-family-moving-L}
In the moving-null formulation, the associate family fixes $A$ and $B$
and rotates $L(z)\mapsto e^{\ii\theta}L(z)$, leaving the Maurer--Cartan
forms $\alpha=A^cA'$ and $\beta=B'B^c$ unchanged. For
$L_\theta=e^{\ii\theta}L$, the moving superminimality expression gives
$\alpha L_\theta+L_\theta'+L_\theta\beta=e^{\ii\theta}(\alpha L+L'+L\beta)$,
so the condition scales by $e^{2\ii\theta}$, consistently with
\Cref{eq:associate-supermin}.
\end{remark}

\begin{remark}[The conjugate surface]
\label{rem:conjugate-surface}
The conjugate minimal surface of $X$ is $X^*:=X_{\pi/2}$, given by
\[
  X^* = c+\re\int\ii\Phi\,\dz = c-\im\int\Phi\,\dz.
\]
In $ALB$-language, this is $X^*=c+\re\int A(\ii L)B\,\dz$. The
holomorphic primitive $F:=\int\Phi\,\dz$ satisfies $F=X+\ii X^*$ up
to a constant. In the normal form $L=1+\ii e_1$, we have
$\ii L=\ii-e_1$, which is again a non-zero null element in the same
$\Sp(1,\C)\times\Sp(1,\C)$-orbit as $L$.
\end{remark}

\section{Polynomial Data and Fixed-Gauge Rigidity}
\label{sec:polynomial}

The aim of this section is to clarify what the
$ALB$-representation says when both spinorial factors are polynomial
and $L$ is kept fixed. All statements are in a chosen fixed-null gauge
and are not invariant under gauge changes in which the null factor is
allowed to move.

For a superminimal surface in the gauge $L=1+\ii e_1$, one of the two
spinorial components is constant, meaning either $[AL]$ or $[LB]$ is
constant. When $A$ and $B$ are polynomial, this projective constancy
forces $AL$ or $LB$ to be actually constant: a nowhere-vanishing
polynomial on $\C$ is constant. 
Polynomial and rational $ALB$-data are especially natural from the
viewpoint of geometric modelling and Pythagorean-hodograph techniques;
see \cite{AltavillaSchroeckerSirVrsek2025}.

\subsection{Stabilizers of the null vector}
\label{subsec:stabilizers-null-vector}

The $ALB$-representation is not unique: even after fixing $L=1+\ii e_1$,
non-trivial stabiliser freedoms remain on both sides of $L$. These are
important in the polynomial setting, because non-constant polynomial
factors may be invisible in the product $ALB$.

With $N_\pm$ as in \Cref{subsec:ODE-parametrization}, one has
$N_-L=LN_+=0$ and $N_-^2=N_+^2=0$. For any holomorphic function $p$,
the elements $U_p:=1+pN_-$ and $V_p:=1+pN_+$ therefore fix $L$ on the
left and right respectively.

\begin{proposition}
\label{prop:stabilizer-freedom}
Let $L=1+\ii e_1$. For every holomorphic function $p$, the maps
\[
  U_p:=1+p(e_3-\ii e_2),\qquad V_p:=1+p(e_3+\ii e_2)
\]
take values in $\Sp(1,\C)$ and satisfy $U_pL=L$ and $LV_p=L$.
Consequently, if $\Phi=ALB$, then also
\[
  \Phi=(AU_p)LB\qquad\text{and}\qquad\Phi=AL(V_pB).
\]
\end{proposition}

\begin{proof}
Let $N_-:=e_3-\ii e_2$ and $N_+:=e_3+\ii e_2$. We compute
\[
  N_-L
  =
  (e_3-\ii e_2)(1+\ii e_1)
  =
  e_3-\ii e_2+\ii e_3e_1+e_2e_1
  =
  e_3-\ii e_2+\ii e_2-e_3
  =
  0.
\]
Similarly,
\[
  LN_+
  =
  (1+\ii e_1)(e_3+\ii e_2)
  =
  e_3+\ii e_2+\ii e_1e_3-e_1e_2
  =
  e_3+\ii e_2-\ii e_2-e_3
  =
  0.
\]
Furthermore,
\[
  N_-^2=(e_3-\ii e_2)^2=0,\qquad N_+^2=(e_3+\ii e_2)^2=0.
\]
Hence
\[
  U_p^s
  =
  (1+pN_-)(1-pN_-)
  =
  1-p^2N_-^2
  =
  1,
\]
and similarly $V_p^s=1$, so $U_p,V_p\in\Sp(1,\C)$.

Finally,
\[
  U_pL=(1+pN_-)L=L+pN_-L=L
  \qquad\text{and}\qquad
  LV_p=L(1+pN_+)=L+pLN_+=L,
\]
and therefore
\[
  (AU_p)LB=A(U_pL)B=ALB
  \qquad\text{and}\qquad
  AL(V_pB)=A(LV_p)B=ALB.\qedhere
\]
\end{proof}

\begin{remark}
\label{rem:stabilizer-invisible-factors}
The factors $U_p$ and $V_p$ belong to the exact stabilisers of $L$
for the left and right actions, so they may be inserted into the
$ALB$-representation without changing $\Phi$ at all. In particular,
non-constant polynomial factors may occur in $A$ or $B$ without
producing any new surface. The geometrically meaningful data is
therefore not the raw pair $(A,B)$, but the projective null curve
$[\Phi]\colon M\to Q^2$ and its spinorial components $[AL]$ and $[LB]$.
\end{remark}

\begin{remark}[Exact versus projective stabilizers]
\label{rem:exact-projective-stabilizers}
The transformations $U_p$ and $V_p$ fix $L$ exactly. One can also
consider the larger projective stabilisers, consisting of elements
preserving the null line $[L]$, i.e., allowing
\[
  UL\in\C^*L\qquad\text{or}\qquad LV\in\C^*L
\]
rather than $UL=L$ or $LV=L$. Their infinitesimal versions are the
Lie algebras $\mathfrak{p}_L^\ell$ and $\mathfrak{p}_L^r$ introduced
in \Cref{rem:intrinsic-pL}. The exact stabilisers are responsible for
gauge redundancy, while the projective stabilisers are responsible for
the superminimality alternatives.
\end{remark}

\subsection{Polynomial normal form}
\label{subsec:polynomial-normal-form}

We now prove the polynomial rigidity statement announced at the
beginning of the section, working throughout in the gauge $L=1+\ii e_1$.

\begin{proposition}[Polynomial normal form in a fixed-null gauge]
\label{prop:polynomial-normal-form}
Let $A,B\colon\C\to\HC$ be polynomial maps satisfying $A^s=B^s=1$,
set $\Phi=ALB$ with $L=1+\ii e_1$, and assume $X=\re\int\Phi\,\dz$
is superminimal. Then, modulo the exact stabiliser freedom of $L$,
the null curve admits one of the two representations
\begin{equation}
\label{eq:polynomial-normal-form}
  \Phi=A_0\,L\,B(z)\qquad\text{or}\qquad\Phi=A(z)\,L\,B_0,
\end{equation}
where $A_0\in\Sp(1,\C)$ or $B_0\in\Sp(1,\C)$ is constant. More
precisely: if $\alpha_2+\ii\alpha_3\equiv0$, then $A$ may be replaced,
without changing $\Phi$, by a constant factor $A_0$; and if
$\beta_2-\ii\beta_3\equiv0$, then $B$ may be replaced by a constant
$B_0$.
\end{proposition}

\begin{proof}
By \Cref{thm:superminimal-ALB-normal-form}, at least one of
$\alpha_2+\ii\alpha_3\equiv0$ or $\beta_2-\ii\beta_3\equiv0$ holds.

Assume $\alpha_2+\ii\alpha_3\equiv0$. By
\Cref{prop:AL-LB-geometric-meaning}, the projective map
\[
  [AL]\colon\C\to\CP^3
\]
is constant. Hence there exist a fixed non-zero null vector $C\in\HC$
and a holomorphic function $h\colon\C\to\C$ with $AL=hC$. Since $A$
is polynomial and $L$ is constant, $AL$ has polynomial components, so
$h$ is polynomial. Moreover, $h$ has no zeros: if $h(z_0)=0$ then
$A(z_0)L=0$, impossible since $A(z_0)\in\Sp(1,\C)$ is invertible and
$L\neq0$. A nowhere-vanishing polynomial on $\C$ is constant, so $AL$
is constant. Setting $A_0:=A(z_0)$ for any $z_0\in\C$, we have
$A_0L=AL$ everywhere, hence $\Phi=ALB=A_0LB$. Writing $U:=A_0^cA$,
\[
  UL = A_0^cAL = A_0^cA_0L = L,
\]
so $U$ takes values in the exact left stabiliser of $L$; the
non-constant part of $A$, if any, is pure stabiliser freedom and is
invisible in $\Phi$.

The proof of the right alternative is completely analogous.
\end{proof}

\begin{remark}
\label{rem:polynomial-rigidity-meaning}
The proposition is a statement modulo stabiliser freedom: a polynomial
factor $A$ may be non-constant even when $AL$ is constant, with its
non-constant part lying in the exact left stabiliser of $L$ and hence
invisible in $\Phi$. The geometrically meaningful datum is the
projective component $[AL]$ (or $[LB]$), which superminimality forces
to be constant.
\end{remark}

\section{Examples}
\label{sec:examples}

We now illustrate the $ALB$-representation by constructing explicit
superminimal surfaces in $\R^4$, working in the fixed-null gauge
$L=1+\ii e_1$. We use the notation
\[
  \eta_-:=e_2-\ii e_3,
  \qquad
  \eta_+:=e_2+\ii e_3,
\]
complementing $N_\pm$ from \Cref{subsec:ODE-parametrization}. One
verifies that $\eta_\pm^2=\eta_\pm^s=0$, $\eta_-L=2\eta_-$, and
$L\eta_+=2\eta_+$, while $N_-L=LN_+=0$, so $N_\pm$ generate exact
stabiliser directions.

\subsection{The two normal-form families}
\label{subsec:normal-form-families}

The classical local normal form of a superminimal null curve is
\[
  \Phi=\mu(1,\ii,g,-\ii g)
  \qquad\text{or}\qquad
  \Phi=\mu(1,\ii,g,\ii g),
\]
where $\mu$ is holomorphic and nowhere zero and $g$ is holomorphic;
see \cite{Bryant1982,DajczerTojeiro2009}. We show how both forms arise
directly from the $ALB$-representation.

Let $M\subset\C$ be simply connected, let $\mu\colon M\to\C^*$ and
$g\colon M\to\C$ be holomorphic. Choose $F$ such that $\mu=e^{-\ii F}$,
and set $R_F:=\cos F+\sin F\,e_1$, so that $R_F^s=1$ and
\[
  R_FL = LR_F = e^{-\ii F}L = \mu L.
\]
Define
\[
  A_g^-:=1+\tfrac{g}{2}\eta_-,
  \qquad
  B_g^+:=1+\tfrac{g}{2}\eta_+;
\]
since $\eta_\pm^2=0$ we have $(A_g^-)^s=(B_g^+)^s=1$.

\smallskip
\noindent\textit{Positive-sign branch.}
Take $A=R_F$ and $B=B_g^+$. Then
$\Phi=ALB=R_FLB_g^+=\mu LB_g^+$.
Using $L\eta_+=2\eta_+$,
\[
  \Phi = \mu(L+g\eta_+) = \mu(1,\ii,g,\ii g),
\]
so $X^+=\re\int\mu(1,\ii,g,\ii g)\,\dz$ is a superminimal conformal
minimal immersion.

\smallskip
\noindent\textit{Negative-sign branch.}
Take $A=A_g^-$ and $B=R_F$. Then
$\Phi=ALB=A_g^-LR_F=\mu A_g^-L$.
Using $\eta_-L=2\eta_-$,
\[
  \Phi = \mu(L+g\eta_-) = \mu(1,\ii,g,-\ii g),
\]
so $X^-=\re\int\mu(1,\ii,g,-\ii g)\,\dz$ is again superminimal.

\begin{remark}
\label{rem:normal-form-families-meaning}
These two constructions recover the classical superminimal branches.
In the positive branch the left spinorial component is projectively
constant:
\[
  [AL]=[R_FL]=[L].
\]
In the negative branch the right spinorial component is projectively
constant:
\[
  [LB]=[LR_F]=[L].
\]
Thus the two $ALB$-alternatives correspond exactly to the two rulings
of $Q^2\simeq\CP^1\times\CP^1$.
\end{remark}

\subsection{Elementary polynomial and transcendental examples}
\label{subsec:elementary-examples}

We now specialise the preceding normal forms.

\begin{example}[Polynomial superminimal surfaces of degree $d$]
\label{ex:polynomial-degree-d}
Take $\mu\equiv1$ and $g(z)=z^{d-1}$ with $d\geq1$. Using the
negative-sign branch gives $\Phi^-=(1,\ii,z^{d-1},-\ii z^{d-1})$,
hence
\[
  X^-(z)
  =
  \re\!\left(z,\;\ii z,\;\frac{z^d}{d},\;-\frac{\ii z^d}{d}\right).
\]
For $d=1$, $g$ is constant, the generalised Gauss map is constant, and
the surface is a plane. For $d=2$,
\[
  X^-(u,v)
  =
  \left(u,\;-v,\;\frac{u^2-v^2}{2},\;uv\right),
\]
the real surface associated with the holomorphic curve
$z\mapsto(z,z^2/2)\in\C^2\simeq\R^4$.
\end{example}

\begin{example}[An exponential superminimal surface]
\label{ex:exponential}
Take $\mu\equiv1$ and $g(z)=e^z$. The negative-sign branch gives
$\Phi^-=(1,\ii,e^z,-\ii e^z)$, and
\[
  X^-(u,v)
  =
  \bigl(u,\;-v,\;e^u\cos v,\;e^u\sin v\bigr).
\]
This is a non-polynomial superminimal conformal minimal immersion.
\end{example}

\subsection{Both spinorial factors non-constant}
\label{subsec:both-factors-nonconstant}

The examples above may be written with one spinorial factor equal to
$1$. We now show that this is not necessary: both $A$ and $B$ may be
non-constant while one projective spinorial component remains constant.

\begin{example}[Both factors non-constant, polynomial Gauss coordinate]
\label{ex:both-factors-nonconstant}

Take $F(z)=z$, $\mu(z)=e^{-\ii z}$, and $g(z)=z$. Use the
positive-sign normal form:
\[
  A(z)=\cos z+\sin z\,e_1,
  \qquad
  B(z)=1+\tfrac{z}{2}\eta_+.
\]
We have
$$
\begin{aligned}
  A(z)^s
  &=
  (\cos z+\sin z\,e_1)(\cos z-\sin z\,e_1) \\
  &=
  \cos^2 z+\sin^2 z =
  1.
\end{aligned}
$$
Moreover, $\eta_+^2=0$, and hence
$$
\begin{aligned}
  B(z)^s
  &=
  \left(1+\frac{z}{2}\eta_+\right)
  \left(1-\frac{z}{2}\eta_+\right) \\
  &=
  1-\frac{z^2}{4}\eta_+^2 =
  1.
\end{aligned}
$$

Then $A^s=B^s=1$
and $A,B\colon\C\to\Sp(1,\C)$. 
Next,
$$
\begin{aligned}
  A(z)L
  &=
  (\cos z+\sin z\,e_1)(1+\ii e_1) \\
  &=
  \cos z+\ii\cos z\,e_1+\sin z\,e_1+\ii\sin z\,e_1^2 \\
  &=
  \cos z-\ii\sin z
  +
  (\ii\cos z+\sin z)e_1 \\
  &=
  e^{-\ii z}(1+\ii e_1) \\
  &=
  e^{-\ii z}L.
\end{aligned}$$
and hence
we have $[A(z)L]=[L]$ for every $z$, so the left spinorial component
of the generalised Gauss map is constant.

We also have that
$$
\begin{aligned}
  L\eta_+
  &=
  (1+\ii e_1)(e_2+\ii e_3) \\
  &=
  e_2+\ii e_3+\ii e_1e_2+\ii^2 e_1e_3 \\
  &=
  e_2+\ii e_3+\ii e_3+e_2 \\
  &=
  2(e_2+\ii e_3) =
  2\eta_+.
\end{aligned}
$$
Therefore
$$
\begin{aligned}
  \Phi
  &=
  ALB
  =
  e^{-\ii z}L\left(1+\frac{z}{2}\eta_+\right) \\
  &=
  e^{-\ii z}\left(L+\frac{z}{2}L\eta_+\right) \\
  &=
  e^{-\ii z}(L+z\eta_+) \\
  &=
  e^{-\ii z}(1,\ii,z,\ii z).
\end{aligned}
$$
Finally,
$$
  \frac{\dd}{\dd z}\bigl(\ii e^{-\ii z}\bigr)=e^{-\ii z},
  \qquad
  \frac{\dd}{\dd z}\bigl(-e^{-\ii z}\bigr)=\ii e^{-\ii z},
$$
and
$$
\begin{aligned}
  \frac{\dd}{\dd z}\bigl(e^{-\ii z}(1+\ii z)\bigr)
  &=
  -\ii e^{-\ii z}(1+\ii z)+\ii e^{-\ii z} =
  z e^{-\ii z}.
\end{aligned}
$$
Thus a primitive of $\Phi$ is
$$
  \bigl(
    \ii e^{-\ii z},\;
    -e^{-\ii z},\;
    e^{-\ii z}(1+\ii z),\;
    \ii e^{-\ii z}(1+\ii z)
  \bigr).
$$
Finally, in order to compute the real and imaginary part, for $z=u+\ii v$, we have that
$
e^{-iz}=e^{-\ii(u+iv)}=e^v\left(\cos u-\ii\sin u\right)
$, hence  $\ii e^{-\ii z}=e^v(\sin u+\ii \cos u)$,
therefore 
$$
\begin{aligned}
    e^{-\ii z}(1+\ii z)&=e^v(\cos u- \ii \sin u)(1-v+\ii u)\\
    &=e^v(\cos u-v\cos u+\ii u \cos u-\ii\sin u+\ii v \sin u +u\sin u)\\
    &=e^v\left((1-v)\cos u+u\sin u+\ii(u\cos u-(1-v)\sin u)\right).
\end{aligned}
$$
We eventually conclude that
\[
  X(u,v)
  =
  e^v
  \begin{pmatrix}
    \sin u\\[2pt]
    -\cos u\\[2pt]
    (1-v)\cos u+u\sin u\\[2pt]
    -u\cos u+(1-v)\sin u
  \end{pmatrix}.
\]
 is a superminimal conformal minimal immersion. Although both
spinorial factors are non-constant, $A$ lies in the projective
stabiliser of $[L]$ (since $A(z)L=e^{-\ii z}L$) but not in the exact
stabiliser of $L$.
\end{example}

\begin{remark}
\label{rem:both-factors-projective-data}
The null curve of \Cref{ex:both-factors-nonconstant} has the same
projective Gauss map as the elementary curve $(1,\ii,z,\ii z)$, since
the two differ only by the scalar $e^{-\ii z}$. The corresponding
minimal surfaces are nevertheless different, as the primitives change.
This illustrates that the projective Gauss data control superminimality,
while the scalar factor $\mu$ affects the actual immersion.
\end{remark}

\begin{example}[Both factors non-constant and non-polynomial]
\label{ex:both-factors-nonpolynomial}
Take $F(z)=z$, $\mu(z)=e^{-\ii z}$, and $g(z)=e^z$. The positive-sign
branch gives
\[
  A(z)=\cos z+\sin z\,e_1,
  \qquad
  B(z)=1+\tfrac{e^z}{2}\eta_+.
\]
Since $A(z)L=e^{-\ii z}L$, we have
$\Phi=ALB=e^{-\ii z}(L+e^z\eta_+)$, with coordinates
\[
  \Phi
  =
  \bigl(e^{-\ii z},\;\ii e^{-\ii z},\;e^{(1-\ii)z},\;\ii e^{(1-\ii)z}\bigr).
\]
A primitive is
$(\ii e^{-\ii z},\,-e^{-\ii z},\,e^{(1-\ii)z}/(1-\ii),\,\ii e^{(1-\ii)z}/(1-\ii))$,
so for $z=u+\ii v$,
\[
  X(u,v)
  =
  \Bigl(
    e^v\sin u,\;
    -e^v\cos u,\;
    \tfrac{e^{u+v}}{2}\bigl(\cos(v-u)-\sin(v-u)\bigr),\;
    -\tfrac{e^{u+v}}{2}\bigl(\cos(v-u)+\sin(v-u)\bigr)
  \Bigr).
\]
\end{example}

\begin{remark}
\label{rem:nonpolynomial-not-polynomial}
The previous example is not reducible to a polynomial example by an
affine change of parameter $z\mapsto az+b$: its projective Gauss
coordinate $g(z)=e^z$ has an essential singularity at $z=\infty$,
whereas polynomial examples have Gauss data meromorphic at infinity.
Thus, with the global parameter $z\in\C$ fixed, this example is
genuinely non-polynomial. (Under arbitrary local holomorphic
reparametrizations, polynomiality is not a meaningful invariant.)
\end{remark}


\begin{thebibliography}{99}

\bibitem{AlarconForstnericLarusson2021}
A. Alarc{\'o}n, F. Forstneri{\v c} and F. L{\'a}russon,
\emph{Holomorphic Legendrian Curves in $\mathbb{CP}^3$ and Superminimal Surfaces in $\mathbb S^4$},
Geometry \& Topology
\textbf{25} (2021), no. 7, 3507--3553.
\href{https://doi.org/10.2140/gt.2021.25.3507}{doi: 10.2140/gt.2021.25.3507}.

\bibitem{AlarconForstnericLopez2021}
A. Alarc{\'o}n, F. Forstneri{\v c} and F. J. L{\'o}pez,
\emph{Minimal Surfaces from a Complex Analytic Viewpoint},
Springer Monographs in Mathematics,
Springer, 2021.
\href{https://doi.org/10.1007/978-3-030-69056-4}{doi: 10.1007/978-3-030-69056-4}.

\bibitem{AltavillaSchroeckerSirVrsek2025}
A. Altavilla, H.-P. Schr{\"o}cker, Z. {\v S}{\'i}r and J. Vr{\v s}ek,
\emph{Minimal Surfaces via Complex Quaternions},
arXiv:2504.17377 [math.CV], 2025.
Accepted for publication in SIAM Journal on Applied Algebra and Geometry.

\bibitem{AHS}
M. F. Atiyah, N. J. Hitchin and I. M. Singer,
\emph{Self-duality in four-dimensional Riemannian geometry},
Proceedings of the Royal Society of London. Series A, Mathematical and Physical Sciences
\textbf{362} (1978), no. 1711, 425--461.
\href{https://doi.org/10.1098/rspa.1978.0143}{doi: 10.1098/rspa.1978.0143}.

\bibitem{Bryant1982}
R. L. Bryant,
\emph{Conformal and minimal immersions of compact surfaces into the $4$-sphere},
Journal of Differential Geometry
\textbf{17} (1982), no. 3, 455--473.
\href{https://doi.org/10.4310/jdg/1214437137}{doi: 10.4310/jdg/1214437137}.

\bibitem{BFLPP}
F. E. Burstall, D. Ferus, K. Leschke, F. Pedit and U. Pinkall,
\emph{Conformal Geometry of Surfaces in $S^4$ and Quaternions},
Lecture Notes in Mathematics, vol. 1772,
Springer, Berlin, 2002.
\href{https://doi.org/10.1007/BFb0109498}{doi: 10.1007/BFb0109498}.

\bibitem{ChiFernandezWu1999}
Q.-S. Chi, L. Fern{\'a}ndez and H. Wu,
\emph{Normalized potentials of minimal surfaces in spheres},
Nagoya Mathematical Journal
\textbf{156} (1999), 187--214.

\bibitem{DajczerTojeiro2009}
M. Dajczer and R. Tojeiro,
\emph{All superconformal surfaces in $\mathbb{R}^4$ in terms of minimal surfaces},
Mathematische Zeitschrift
\textbf{261} (2009), no. 4, 869--890.
\href{https://doi.org/10.1007/s00209-008-0355-0}{doi: 10.1007/s00209-008-0355-0}.

\bibitem{EellsSalamon1985}
J. Eells and S. Salamon,
\emph{Twistorial construction of harmonic maps of surfaces into four-manifolds},
Annali della Scuola Normale Superiore di Pisa. Classe di Scienze. Serie IV
\textbf{12} (1985), no. 4, 589--640.


\bibitem{Forster1981}
O. Forster,
\emph{Lectures on Riemann Surfaces},
Graduate Texts in Mathematics, vol. 81,
Springer, New York, 1981.
\href{https://doi.org/10.1007/978-1-4612-5965-9}{doi: 10.1007/978-1-4612-5965-9}.

\bibitem{Forstneric2020CY}
F. Forstneri{\v c},
\emph{The Calabi--Yau Property of Superminimal Surfaces in Self-Dual Einstein Four-Manifolds},
The Journal of Geometric Analysis
\textbf{31} (2021), no. 5, 4754--4780.
\href{https://doi.org/10.1007/s12220-020-00455-6}{doi: 10.1007/s12220-020-00455-6}.

\bibitem{Friedrich1997}
T. Friedrich,
\emph{On superminimal surfaces},
Archivum Mathematicum (Brno)
\textbf{33} (1997), no. 1--2, 41--56.

\bibitem{GuadalupeRodriguez1983}
I. V. Guadalupe and L. Rodriguez,
\emph{Normal curvature of surfaces in space forms},
Pacific Journal of Mathematics
\textbf{106} (1983), no. 1, 95--103.
\href{https://doi.org/10.2140/pjm.1983.106.95}{doi: 10.2140/pjm.1983.106.95}.

\bibitem{GunningRossi1965}
R. C. Gunning and H. Rossi,
\emph{Analytic Functions of Several Complex Variables},
Prentice-Hall, 1965.

\bibitem{HarveyLawson1982}
R. Harvey and H. B. Lawson,
\emph{Calibrated geometries},
Acta Mathematica
\textbf{148} (1982), 47--157.
\href{https://doi.org/10.1007/BF02392726}{doi: 10.1007/BF02392726}.

\bibitem{HoffmanOsserman1980}
D. A. Hoffman and R. Osserman,
\emph{The Geometry of the Generalized Gauss Map},
Memoirs of the American Mathematical Society, vol. 28, no. 236,
American Mathematical Society, Providence, RI, 1980.
\href{https://doi.org/10.1090/memo/0236}{doi: 10.1090/memo/0236}.



\bibitem{Osserman1986}
R. Osserman,
\emph{A Survey of Minimal Surfaces},
2nd ed.,
Dover Publications, New York, 1986.

\bibitem{Wintgen1979}
P. Wintgen,
\emph{Sur l'in{\'e}galit{\'e} de Chen--Willmore},
Comptes Rendus de l'Acad{'e}mie des Sciences de Paris. S{'e}rie A--B
\textbf{288} (1979), no. 21, A993--A995.

\end{thebibliography}
\end{document}